\documentclass[accepted]{uai2026} 
                        
\usepackage[american]{babel}
\usepackage{tikz}
\usepackage{amssymb}
\usepackage{amsthm}
\newcommand{\train}{\mathcal{D}}
\newcommand{\obsdata}{\train_{\text{obs}}}
\newcommand{\mixdata}{\train_{\text{mix}}}
\newcommand{\tsone}{\tau_1^{*2}}
\newcommand{\tstwo}{\tau_2^{*2}}
\newcommand{\intdata}{\train_{\text{int}}}
\def\transpose{\intercal}
\newtheorem{lemma}{Lemma}[section]
\newtheoremstyle{boldremark}
  {3pt}                
  {3pt}                
  {\normalfont}        
  {}                   
  {\bfseries}          
  {:}                  
  {.5em}               
  {}                   

\theoremstyle{boldremark}
\newtheorem{remark}{Remark}
\newtheorem{theorem}{Theorem}[section]

\usepackage{natbib} 
\usepackage{mathtools} 
\usepackage{amsmath}
\usepackage{booktabs} 
\usepackage{tikz} 

\title{The relative value of interventional and observational samples in Bayesian Causal Linear Gaussian Models}

\author[1]{Valentinian Lungu}
\author[2]{Anish Dhir}
\author[3]{Mark van der Wilk}
\author[1]{Ioannis Kontoyiannis}

\affil[1]{Statistical Laboratory, Centre for Mathematical Sciences, University of Cambridge}
\affil[2]{Imperial College London}
\affil[3]{University of Oxford}
  
\begin{document}
\maketitle

\begin{abstract}
We investigate the asymptotic properties of Bayesian bivariate causal discovery for Gaussian Linear Structural Equation Models (SEMs) with heteroscedastic noise. We demonstrate that with purely observational data, the posterior distribution over the models fails to consistently identify the true causal structure—a consequence of the fundamental non-identifiability within the Markov Equivalence Class. Specifically, if the true generating mechanism corresponds to a connected graph ($A \rightarrow B$ or $B \rightarrow A$), the asymptotic behavior of the posterior is given by the ratio between the prior on the true model and the push-forward prior of the alternative. In contrast, for the independence model, we establish that the posterior concentrates at a stochastic polynomial rate of $O_p(n^{-1/2})$. 

To resolve this non-identifiability, we incorporate $m$ interventional samples and characterize the concentration rates as a function of the observational-to-total sample ratio, $\eta$. We identify a sharp \textit{concentration dichotomy}: while the independence graph maintains a polynomial $O_p(N^{-1/2})$ rate (where $N = n+m$), connected graphs undergo a phase transition to exponentially fast convergence. This highlights an \textit{exponential} relative importance between the two data types, as altering the amount of one data type directly changes the exponent governing the concentration speed. We derive explicit formulae for the exponential decay rates and provide precise conditions under which mixing observational and interventional data optimizes concentration speed. Finally, our theoretical findings are validated through empirical simulations in Bayesian Gaussian equivalent (BGe)-style prior specifications offering a principled foundation for experimental design in Bayesian causal discovery.

\end{abstract}

\section{Introduction}\label{sec:intro}

Over the last several decades, there has been a significant surge of interest in developing robust and scalable algorithms for detecting hidden dependencies between variables. In the causal modeling framework, these variables are typically represented as nodes in a Directed Acyclic Graph (DAG), where directed edges represent causal mechanisms. The researcher's objective is to perform \textit{structure learning}---identifying the graph that best explains the observed data. This problem of causal graph identification is foundational across diverse domains, including the mapping of gene regulatory networks~\citet{THOR}, the evaluation of clinical treatments~\citet{wu:24}, and the optimization of advertising strategies~\citet{VANDENBROECK2018470}.

There is a rich set of causal discovery tools. These include constraint-based methods like the PC algorithm~\citet{PCalgorithm}, score-based optimization approaches~\citet{Peters_2013,lasso,NOTEARS}, and Bayesian methodologies~\citet{bayesdag,BCDNets,beingBayesian,annadani2021variational,GFlowbayesian,dynamicbayesian,cao:19}. More recently, meta-learning frameworks based on Neural Processes have also been explored in ~\citet{dhir2025metalearningapproachbayesiancausal}and \citet{dhir2026estimatinginterventionaldistributionsuncertain}, offering new avenues for amortized causal inference. While optimization-based methods typically yield a point estimate of the causal graph, Bayesian/amortized approaches are usually preferred because they provide a posterior distribution over the graph space. This uncertainty quantification is particularly vital in safety-critical domains, where distinguishing between strong causal evidence and mere lack of data is essential for robust decision-making. Although these works show good prediction power, a deep theoretical approach is still missing. 

In this work, we focus on the Bayesian theoretical treatment of \textit{Gaussian Linear Structural Equation Models (SEMs)}. These linear models have been rigorously studied from both  frequentist~\citet{peters2014identifiability,pearl2009causality} and Bayesian~\citet{heckerman2006a} perspectives. Notably, \citet{peters2014identifiability} demonstrated that if Gaussian errors are homoscedastic (sharing a common variance) or if variances are known, the causal structure is identifiable from purely observational data. Conversely, under unknown heteroscedastic noise, the model is identifiable only up to its \textit{Markov Equivalence Class (MEC)}. From a Bayesian standpoint, \citet{heckerman2006a} introduced a specific class of priors that satisfy \textit{score equivalence}, ensuring the posterior assigns equal mass to all structures within the same MEC. This is commonly known as the \textit{Bayesian Gaussian equivalent (BGe)} score. 

This paper presents a two-fold investigation into posterior behavior in non-identifiable and interventional settings. First, we examine the bivariate Gaussian SEM with heteroscedastic noise. We demonstrate that when a causal edge exists between two nodes, the posterior of the true model fails to concentrate to $1$ as the sample size $n$ grows large. Instead, the posterior limit is determined by the ratio between the prior on the true model and the \textit{push-forward prior} of the alternative model. This result generalizes the findings by \citet{dhir2024bivariatecausaldiscoveryusing} and \citet{geiger_annals}, which showed that a ratio of $1$ leads to an uninformative posterior split between the two causal directions. In contrast, when the independence model is the true generating mechanism, we find that the posterior concentrates to $1$ at a stochastic rate of $O_p(1/\sqrt{n})$. 

To resolve such non-identifiability, researchers often leverage \textit{hard interventions} to distinguish between direct causation and latent confounding~\citet{pearl2009causality}. However, in many practical scenarios---such as gene knockout experiments or clinical trials---interventions are costly or technically difficult, while observational data is abundant. This creates a critical trade-off: determining how many expensive interventional samples are required to augment cheap observational data to achieve identifiability. By forcing a variable to a fixed value, an intervention ``changes'' the graph, severing its natural dependencies and isolating its downstream effects.

The second part of this work analyzes the scenario where the researcher has access to $m$ interventional samples alongside $n$ observational samples. While the sample complexity of mixing observational and interventional data has been recently studied as a \textit{closeness testing} problem for discrete~\citet{acharya23b} and continuous~\citet{jamshidi2025samplecomplexitynonparametriccloseness} data, we provide the first comprehensive theoretical investigation into \textit{posterior concentration} as a function of the observational sample ratio $\eta = n/N$, where $N = n+m$. We identify a second dichotomy in concentration speeds similar to the one observed in the homoscedastic setting described by \cite{lungu2025bayesiancausaldiscoveryposterior}: for the connected graphs, the posterior concentrates to $1$ at an \textit{exponential rate}, with an exponent characterized by a weighted sum of \textit{relative entropies} (Kullback-Leibler divergences) between the observational and interventional densities. This result implies an  \textit{exponential relative dependence}  between the two data regimes; specifically, the sample composition directly determines the exponent of the convergence rate, such that even marginal shifts in the data mixture result in exponential-scale changes to the speed of identification. Our analysis distinguishes between interventions on the cause versus on the effect. While intervening on the effect intuitively benefits from a mix of data, one might assume that intervening on the cause renders observational data redundant. Counter-intuitively, we show that for \textit{moderate SNR regimes and conservative interventions}, observational samples remain important. In the latter setting, the interventional signal alone may be insufficient to distinguish the causal mechanism from the noise floor; observational data helps by allowing the posterior to calibrate the baseline variance and more effectively "learn" the structure.   In contrast to this exponential regime, we establish that for the independence model, the posterior convergence undergoes a phase transition, slowing to a stochastic polynomial rate of $O_p(1/\sqrt{N})$. 


\section{Related work}\label{sec:related_work}
The problem of identifying generating structures from observational data has been a central focus of the SEM literature. Early constraint-based approaches, such as the PC and FCI algorithms \citet{spirtes:00}, leverage conditional independence tests to recover the causal graph. However, because these methods rely on the faithfulness assumption, they are generally limited to recovering the Markov Equivalence Class (MEC) rather than the unique Directed Acyclic Graph (DAG). Along the same lines, score-based methods were developed to search for graphs that maximize a scoring function, often derived from parametric assumptions \citet{Chickering02a, heckerman2006a}. 

In the linear-Gaussian setting, the Bayesian Gaussian equivalent (BGe) score is a standard choice \citet{geiger2021learninggaussiannetworks}. While computationally attractive, linear-Gaussian models suffer from a fundamental limitation known as the "Gaussian Trap": the covariance matrix of the observed variables generally does not contain sufficient information to distinguish between Markov equivalent graphs (e.g., $A \rightarrow B$ vs. $B \rightarrow A$) \citet{shimizu06a}. Exceptions exist, such as the result by \citet{peters2014identifiability}, which proves identifiability in the special case where all error terms share equal variance. The "Gaussian Trap" is, however, an artifact of the distributional assumption. The introduction of the Linear Non-Gaussian Acyclic Model (LiNGAM) by \citet{shimizu06a} demonstrated that if error terms are non-Gaussian, the full causal structure is identifiable using Independent Component Analysis (ICA) and the Darmois-Skitovich theorem. Building on this, \citet{hoyer2012bayesiandiscoverylinearacyclic} proposed Bayesian LiNGAM to combine the identifiability of non-Gaussian models with the uncertainty quantification of the Bayesian paradigm. Rather than outputting a single point-estimate graph, Bayesian LiNGAM computes the posterior probability over DAGs, using non-Gaussian likelihoods to break the symmetry of the MEC.

While non-Gaussianity provides a statistical route to identifiability, interventions offer a physical route. Interventional data fundamentally alters the search space by refining the MEC into the smaller Interventional Markov Equivalence Class (I-MEC) \citet{hauser2012characterizationgreedylearninginterventional}. Interventions are typically categorized as "hard" or "soft." A hard intervention, denoted by $do(A=a)$, forces a variable to a specific value, effectively removing all incoming edges to $A$. In contrast, a soft (or parametric) intervention modifies the conditional distribution of $A$—for instance, by shifting the noise mean or variance—without changing the dependency on its parents. \citet{hauser2012characterizationgreedylearninginterventional} characterized the I-MEC for hard interventions, showing that as the set of intervention targets expands, the I-MEC shrinks, eventually collapsing to the unique true DAG.

Recent interest has shifted focus from asymptotic identifiability to finite-sample complexity, addressing the practical trade-offs between data types. \citet{acharya23b} investigate the sample complexity of distinguishing cause from effect in discrete bivariate settings.  Their work establishes tight bounds on the number of interventional samples ($m$) needed as a function of available observational samples ($n$).

In the continuous domain, \citet{jamshidi2025samplecomplexitynonparametriccloseness} address the problem of 
detecting causal direction/hidden confounding from interventional and observational data. They frame it as a non-parametric closeness testing problem. Estimating divergence measures like the KL divergence in continuous, high-dimensional spaces is notoriously difficult due to estimation bias. To mitigate this, the authors propose an estimator based on the Von Mises expansion, establishing the first sample complexity guarantees for distinguishing cause from effect in continuous, non-linear systems with potential hidden confounders. 

\section{Non-identifiable setting}
\label{sec:non_ident}

In this work, we consider the 2-nodes linear structural equations model

\begin{equation}
    X = A^\transpose X + \epsilon,
    \label{eq:SCM_vector}
\end{equation}
where $X = [X(1), X(2)]^\transpose \in \mathbb{R}^2$ is a 2-dimensional vector, and $\epsilon = [\epsilon(1), \epsilon(2)]^\transpose \in \mathbb{R}^2$ is joint Gaussian noise with zero-mean and diagonal covariance matrix $\Sigma = Diag(\tau_1^2, \tau_2^2)$. The matrix A is described in terms of its structure $S$ and weight $w$, where $S$ is a binary matrix with entries $S_{ij} = \mathbb{I}\{A_{ij} \neq 0\}$. Naturally, we refer to $S$ as being the structure matrix and the parameters $\theta := [w, \tau_1^2, \tau_2^2]^\transpose \in \mathbb{R} \times (0, \infty)^2=:\Theta$ as the associated parameters. 

Throughout the article we assume that $S$ is the corresponding adjacency matrix of a 
DAG on the set $\{1, 2\}$. Since we consider only 2 nodes, $S$ can take values only in the following set $\mathcal{S} := \{S^1, S^2, S^3\}$ where

$$S^1 = \begin{bmatrix}
    0 & 1\\
    0 & 0
\end{bmatrix}, \quad S^2 = \begin{bmatrix}
    0&0\\
    1&0
\end{bmatrix}, \quad \text{and}\quad S^3 = \begin{bmatrix}
    0&0\\
    0&0
\end{bmatrix}.$$
For simplicity, we write $A = m(S, w)$ for the map and denote $P_{m(S, \theta)}$ as the law of $X = (I-A^\transpose)^{-1}\epsilon$ with the covariance matrix given by the last two components of $\theta$. The corresponding density for $P_{m(S, \theta)}$ is denoted by $f(x\mid \theta, S)$ and the log-likelihood by $\ell(\theta\mid S) := \log f(x\mid \theta, S)$ with $x = [x(1), x(2)]^\transpose$ for all $S \in \mathcal{S}$. For each structure matrix $S^i$, and the corresponding parameters $\theta^i$ we define the model $\mathcal{M}^i$ as follows

$$\mathcal{M}^i = \{\theta^i, S^i\},\quad \text{for } i \in \{1,2,3\}.$$

\subsection{Frequentist approach}

\begin{lemma}
\label{lemma:non_ident}
    The two structures, $S^1$ and $S^2$ are non-identifiable, i.e. for any choice of parameters $\theta^1\in \Theta$, there exists $\theta^2\in \Theta$ such that the likelihoods of the two models are the same, i.e. $f(x\mid \theta^1, S^1) = f(x\mid \theta^2, S^2)$ for any $x = [x(1), x(2)] \in \mathbb{R}^2$.
\end{lemma}

We denote by $\gamma: \Theta \to \Theta$ the map defined by the parameter transformation described in the proof (see Appendix \ref{appendix:lemma_non_ident}) of the above lemma . Specifically, if $\theta^1 = [w, \tau_1^2, \tau_2^2]^\transpose$, then the mapping $\theta^2 = \gamma(\theta^1)$ is given by
\begin{equation}
\label{eq:gamma}
\gamma(w, \tau_1^2, \tau_2^2) = \left[
\frac{w \tau_2^2}{w^2 \tau_2^2 + \tau_1^2}, \
w^2 \tau_2^2 + \tau_1^2, \
\frac{\tau_1^2 \tau_2^2}{w^2 \tau_2^2 + \tau_1^2}
\right]^\transpose
\end{equation}
This mapping formally establishes the non-identifiability result. Since $\gamma(\theta^1)$ is a continuous transformation, and we have shown that for any choice of parameters $\theta^1$ there exists a distinct set of parameters $\theta^2 = \gamma(\theta^1)$ such that $f(x \mid S^1, \theta^1) = f(x \mid S^2, \theta^2)$ for any $x\in \mathbb{R}^2$, the two structures $S^1$ and $S^2$ are non-identifiable.

Given $n$ independent and identically distributed (i.i.d.) observations $\obsdata := (X_i)_{1\leq i\leq n}$ generated from \eqref{eq:SCM_vector}, we calculate the maximum likelihood estimators (MLEs) for the three structures considered: $\hat{\theta}^1$ for $S^1$ (where $X(1)\to X(2)$),  $\hat{\theta}^2$ for $S^2$ (where $X(2) \to X(1)$) and $\hat{\theta}^3$ for the independence model $S^3$.
Lemma~\ref{lemma:MLE_nonident} establishes that $\hat{\theta}_2 = \gamma(\hat{\theta}_1)$, confirming that $S^1$ and $S^2$ are observationally equivalent. This symmetry implies that the causal structures are non-identifiable through likelihood maximization alone, as both models achieve the same global maximum likelihood.

\subsection{Bayesian approach}
We adopt a Bayesian framework to perform causal structure inference. While $S^1$ and $S^2$ remain observationally non-identifiable, we characterize the asymptotic behavior of the posterior distribution and discuss the impact of the prior distributions.

We define a hierarchical model over the set of causal structures by first assigning a uniform prior distribution over the model space, $\pi(S) = \frac{1}{3}$, for $S \in \mathcal{S}$. Conditional on the structure $S$, we assign a prior density $\pi(\theta \mid S)$ to the associated parameter vector $\theta$. For the purpose of this general formulation, we leave the functional form of $\pi(\theta \mid S)$ unspecified for now.

Assuming access to $n$ i.i.d. samples drawn from \eqref{eq:SCM_vector}, $\obsdata = (X_i)_{1\leq i\leq n}$, the posterior over the structures is
$S \in \mathcal{S}$
\begin{equation}
    \pi(S\mid\obsdata) = \frac{\int_{\theta \in \Theta} f(\obsdata\mid S, \theta)\pi(\theta\mid S)d\theta}{\sum_{j=1}^3 \int_{\theta \in \Theta} f(\obsdata\mid S^j, \theta)\pi(\theta\mid S^j)d\theta}. 
    \label{eq:post_observ}
\end{equation}

The marginal likelihood, 
$\int_{\theta \in \Theta} f(\obsdata \mid S, \theta)\pi(\theta \mid S)d\theta$,
is inherently sensitive to the choice of prior $\pi(\theta \mid S)$. However, for sufficiently regular priors, the integral is well-approximated via Laplace's method as $n \to \infty$. Due to space constraints, we defer the formal regularity conditions to the Appendix in Lemma~\ref{lemma:laplace_regular} where we verify that our setting satisfies the Laplace regularity requirements described by \citep{Kass1990TheVO}.

\begin{theorem}
\label{thm:nonident_mec}
    Let $\obsdata = (X_i)_{1\leq i \leq n}$ be $n$ i.i.d. observations generated by \eqref{eq:SCM_vector}. Assume that the priors $\pi(\theta\mid S)$ are four-times differentiable for any $S \in \mathcal{S}$. If the true generating model is $\mathcal{M}^1 = \{\theta^*, S^{1*}\}$ with $\theta^* = [w^*, \tsone, \tstwo]^\transpose$ and $|w^*| >0$, then, as $n\to\infty$, 

    $$\pi(S^{1*}\mid\obsdata) \xrightarrow{a.s.} \frac{1}{1+\frac{\tstwo}{w^{*2}\tstwo+\tsone}\cdot\frac{\pi(\theta^*\mid S^2)}{\pi(\theta^*\mid S^{1*})}}. $$
    Moreover, we note that 
    $$\frac{\tstwo}{w^{*2}\tstwo+\tsone}\cdot\frac{\pi(\theta^*\mid S^2)}{\pi(\theta^*\mid S^1)} = \frac{\gamma^{-1}\#\pi(\theta^*\mid S^2)}{\pi(\theta^*\mid S^{1})}$$
    
    with $\gamma^{-1}\#\pi(\cdot\mid S^2)$ being the push-forward measure of the alternative prior (for $S^2$) under the transformation defined in \eqref{eq:gamma}.
\end{theorem}

\begin{remark}
    Assume the conditions in Theorem~\ref{thm:nonident_mec} hold. If $\theta^*$ are the true parameters under $S^1$, then $\gamma(\theta^*)$ are the corresponding parameters under $S^2$ that yield the same likelihood. If $S^2$ is considered the true generating mechanism, one can show similarly that as $n \to \infty$
\begin{align*}
    \pi(S_2 \mid \obsdata) &\xrightarrow{a.s.} \frac{1}{1 + \frac{w^{*2}\tau_2^{*2} + \tau_1^{*2}}{\tau_2^{*2}} \cdot \frac{\pi(\theta^* \mid S_1)}{\pi(\theta^* \mid S_2)}},
\end{align*}
which implies that $\pi(S^1\mid \obsdata) +\pi(S^2 \mid \obsdata) \xrightarrow{a.s.}1.$
\end{remark} 

\begin{remark}
    Although structures $S^1$ and $S^2$ are observationally non-identifiable, the posterior probability of structure $S^1$ does not necessarily converge to $1/2$. This might suggest the existence of a decision rule for detecting the true generating structure; for instance, one could select $S^1$ if its posterior probability exceeds $1/2$, and $S^2$ otherwise. At first glance, this appears to contradict the non-identifiability of the model. However, we note that the true generating parameters $\theta^*$ are unknown to the observer. Consequently, without knowledge of the specific parameter values, one cannot determine whether the threshold for the posterior ratio should be set above or below $1/2$ to correctly identify the true graph. 
\end{remark}

\begin{remark}
    This result also shows that the posterior odds ratio converges to a constant value which is dependent only on the priors placed on the two structures $S^1$ and $S^2$. It also confirms the result of \citet{dhir2024bivariatecausaldiscoveryusing} which shows theoretically that if the posterior of two models are equal, then the selected priors must satisfy the following relation
$\pi(\theta^*\mid S^{1*}) = \gamma^{-1}\#\pi(\theta^*\mid S^2)$. Moreover, we note that this ratio is also central to the derivation of the Bayesian Gaussian equivalent (BGe) score in the case of two nodes which was studied in \citet{geiger1998characterization}. In fact, they show that if the global parameter independence (or the independent causal mechanism -- ICM) assumption holds, i.e. the two priors factorise as follows $\pi(w, \tau_1^2, \tau_2^2 \mid S^1) = \pi(w, \tau_1^2\mid S^1)\pi(\tau_2^2\mid S_1)$ and $\pi(w, \tau_1^2, \tau_2^2 \mid S^2) = \pi(w, \tau_2^2\mid S^2)\pi(\tau_1^2\mid S_2)$, then the ratio is one

$$\frac{\gamma^{-1}\#\pi(\theta\mid S^2)}{\pi(\theta\mid S^1)} =   1$$
if and only if $[w, \tau_1^2, \tau_2^2]^\transpose$ follow a normal-Wishart prior under the two models $S^1$ and $S^2$. This case is further discussed in Section \ref{subsec:bge_models}.
\end{remark}

Next, we consider the scenario where the true generating mechanism is the independence model $S^3$. In this case, the following theorem demonstrates that the posterior probability of the true model concentrates to $1$ asymptotically. Notably, however, this convergence is not almost sure (as observed in the connected cases), but is instead characterized by a rate of $O_p(n^{-1/2})$, where $O_p(\cdot)$ denotes boundedness in probability. This suggests a significant slow-down in concentration compared to the exponential rates typical of identifiable causal directions derived in \citet{lungu2025bayesiancausaldiscoveryposterior}or in the next sections.

\begin{theorem}
    \label{thm:non_ident_S3}
    Assume the conditions from Theorem \ref{thm:nonident_mec}. If the true generating model is $\mathcal{M}^3 = \{\theta^*, S^{3*}\}$ with $\theta^* = [0, \tsone, \tstwo]^\transpose$, as $n\to\infty$, we have,

    $$\sqrt{n}(1-\pi(S^{3*}\mid \obsdata)) = O_p(1).$$
\end{theorem}

\begin{remark}
\label{remark:log_posterior_chi}
    The above result indicates that the posterior recovers the true generating structure asymptotically. Furthermore, one can derive the asymptotic distribution of the log-posterior odds from our proof. For $S^i$ with $i \in \{1,2\}$, then
    \begin{equation}
        \label{eq:conv_to_chi_sq}
        2\log\left(\sqrt{n}\frac{\pi(S^i \mid \obsdata)}{\pi(S^{3*} \mid \obsdata) }\right) + \log \frac{\pi(\theta^* \mid S^i)}{\pi(\theta^* \mid S^{3*})} \xrightarrow{d} \chi_1^2,
    \end{equation}
    where $\theta^* = [0, \tsone, \tstwo]^\transpose$, $\xrightarrow{d}$ denotes convergence in distribution and $\chi_1^2$ is a chi-squared distribution with one degree of freedom.
\end{remark}

\begin{remark}
    The aforementioned results extend naturally to the general $d$-node setting. In the case of the empty DAG (the independence model), the posterior concentrates on the true structure at a rate of $O_p(n^{-1/2})$. For all other cases, the posterior distribution concentrates on the MEC of the true generating structure. Notably, the posterior mass is not necessarily distributed uniformly within the MEC; rather, the relative weights of equivalent structures depend on the specific prior densities considered.
\end{remark}

\section{Identifiability via Interventional Data}
\label{sec:ident}
As established in the previous section, the posterior distribution over the model space fails to concentrate on the true generating structure when restricting observations to the observational regime, owing to the Markov equivalence of structures $S^1$ and $S^2$. To resolve this non-identifiability problem, we extend the data setting to include interventional data.

We assume access to a heterogeneous dataset consisting of an observational component, $\obsdata := (X_i)_{1\leq i \leq n} $ with $X\stackrel{\text{i.i.d.}}{\sim} P_{m(S,\theta)}$, and an interventional component, $\intdata := (Y_j)_{1\leq j \leq m}$. To formalize the latter, we use the $do(\cdot)$ operator introduced in \cite{pearl2009causality}. An intervention $do(X(j) = y)$ corresponds to a structural modification of the data-generating process wherein the value of node $j$ is fixed to $y$, effectively removing all incoming causal edges to $j$ while preserving the structural equations of its descendants.

We consider a fixed interventional design targeting the second variable. Let the intervention be defined as $do(X(2) = y)$ for a fixed $y \in \mathbb{R}$. The samples in the interventional dataset are denoted by $Y_i = [Y_i(1), y]^\transpose$ for $i = 1, \dots, m$. The marginal distribution of the non-intervened variable $Y(1)$ allows us to discriminate between the two connected causal structures,  $S^1$ and $S^2$.

If the true generating mechanism is $S^1$ (implying the causal direction $X(2) \to X(1)$), the intervention on $X(2)$ propagates to $X(1)$. Specifically, for parameters $\theta^* = [w^*, \tsone, \tstwo]^\transpose$:
\begin{equation}
    Y(1) \mid \{ do(Y(2) = y), \theta^*, S^1\} \sim N(w^*y, \tsone).
\end{equation}
Conversely, if the true mechanism is $S^2$ (where $X(1) \to X(2)$) or the independence model $S^3$, the intervention on $X(2)$ breaks the causal dependency (in the case of $S^2$) or no dependency exists (in $S^3$). In both cases, the distribution of $Y(1)$ remains governed by its marginal parameters
\begin{equation}
    Y(1) \mid \{do(Y(2) = y), \theta^*,  S^i\} \sim N(0, \tsone),
\end{equation}
for $i \in \{2, 3\}$. This distributional difference renders $S^1$ and $S^2$ distinguishable. To analyze the asymptotic behavior of the posterior $\pi(S \mid \mixdata)$ with $\mixdata :=\obsdata\cup\intdata$, we characterize the data regime by the fraction of observational samples.

To formalize this setting, let $N$ be the total number of samples, with $n_N$ observational and $m_N$ interventional samples. Assuming both sequences, $\{n_N\}_N$ and $\{m_N\}_N$, increase with $N$, we define the asymptotic observational ratio $\eta \in (0,1)$ as
\begin{equation}
    \label{eq:def_eta}
    \eta_N := \frac{n_N}{N}, \quad \text{where} \quad \lim_{N \to \infty} \eta_N = \eta.
\end{equation}

The following analysis investigates the concentration rate of the posterior as a function of the mixing parameter $\eta$ and the true generating mechanism. First, we consider the connected cases, i.e. $S^1$ and $S_2$, and show that in this setting, the posterior concentrates almost surely to $1$ at an exponential rate. 

\begin{theorem}
    \label{thm:ident_S1}
    Let $\obsdata = (X_i)_{1\leq i \leq n_N}$ be $n_N$ i.i.d. samples generated by \eqref{eq:SCM_vector} and $\intdata = (Y_i)_{1\leq i\leq m_N}$ be $m_N$ i.i.d. samples generated by interventional model, when intervened upon the second node, $\cdot\mid do(Y(2) = y )$. Assume that the priors $\pi(\theta\mid S)$ are four-times differentiable for $S \in \mathcal{S}$. If the true generating model is $\mathcal{M}^1 = \{\theta^*, S^{1*}\}$ with $\theta^*:= [w^*, \tsone, \tstwo]^\transpose$, as $N \to \infty$, we have,

    \[
        \frac{1}{N} \log \left( \frac{1}{\pi(S^{1*} \mid \mixdata)} - 1 \right) \xrightarrow{a.s.} -D_{12}(\eta),
    \]
    with 
\begin{equation}
    \label{eq:D_12_def}
    D_{12}(\eta) := \frac{1}{2} \log \left( \frac{\eta \sigma_{1,X}^2 + (1-\eta) \sigma_{1,Y}^2}{(\sigma_{1,X}^2)^\eta (\tsone)^{\bar{\eta}}} \right)
\end{equation}
    and  $\sigma_{1,X}^2:=w^{*2}\tstwo+\tsone,$ $\sigma_{1, Y}^2 := w^{*2}y^2+\tsone$.

    Conversely, if the true generating model is $\mathcal{M}^2 = \{\theta^*, S^{2*}\}, $ then the posterior concentrates to 1, but at a different rate, i.e. as $N \to \infty$

    \[
        \frac{1}{N} \log \left( \frac{1}{\pi(S^{2*} \mid \mixdata)} - 1 \right) \xrightarrow{a.s.} -D_{21}(\eta),
    \]
    with 
    \begin{equation}
        \label{eq:D_21_def}
        D_{21}(\eta) = \frac{1}{2}\log \left( 1- \frac{\eta^2 w^{*2}\tau_1^{*2}}{\eta\sigma_{2, Y}^2 +\bar{\eta}y^2}\right) \left(\frac{\sigma_{2, Y}^2}{\tau_2^{*2}} \right)^ \frac{\eta}{2},
    \end{equation}
with $\sigma_{2, Y}^2 = w^{*2}\tsone+\tstwo$.
\end{theorem}

The next two lemmas describe the properties of the $D_{12}(\eta)$ and $D_{21}(\eta)$ exponents. The implications of these results are further discussed in the the next section.

\begin{lemma}
\label{lemma:concavity_D_12}
The exponent $D_{12}(\eta)$ defined in \eqref{eq:D_12_def} is positive and strictly concave for $\eta \in (0,1)$ if $\tstwo \neq y^2 >0$. Moreover, if 

\begin{equation}
    \label{eq:cond_D_12}
    \frac{w^{*2}(\tstwo-y^2)}{w^{*2}y^2+\tsone} > \log \left( 1+\frac{w^{*2}\tstwo}{\tsone} \right),
\end{equation}
$D_{12}(\eta)$ achieves a \textit{unique} maximum for some $\eta^* \in (0,1)$. Otherwise, $D_{12}(\eta)$ is strictly decreasing for $\eta \in (0,1)$.
\end{lemma}

\begin{lemma}
    \label{lemma:D_21_concavity}
    The exponent $D_{21}(\eta)$ defined in \eqref{eq:D_21_def} is positive and strictly concave for $\eta \in (0,1)$. $D_{21}(\eta)$ achieves a unique maximum for $\eta^* \in (0,1)$.
\end{lemma}

Analogous to Theorem~\ref{thm:non_ident_S3}, the posterior achieves a stochastic polynomial concentration rate of $O_p(N^{-1/2})$. We formalize this below.

\begin{theorem}
\label{thm:ident_S3}
    Assume the conditions in Theorem \ref{thm:ident_S1} hold. If the true generating model is $\mathcal{M}^3 = \{\theta^*, S^{3*}\}$ with $\theta^*:= [0, \tsone, \tstwo]^\transpose$, as $N \to \infty$, we have,
    $$\sqrt{N}(1-\pi(S^{3*}\mid \mixdata)) = O_p(1).$$
\end{theorem}

\section{Interpretation and Numerical Study}
\label{sec:interpretation}
In this section, we investigate the behavior of the asymptotic convergence rates as a function of the mixing parameter $\eta$. Furthermore, we illustrate the theoretical findings through an empirical study on a 2-node causal model with a prior structure closely related to the one used for the Bayesian Gaussian equivalent (BGe) models.

\subsection{Analysis of Convergence Rates}
\label{subsec:numerical_analysis_conv_rates}
Theorem \ref{thm:ident_S1} established that if the true generating mechanism is $S^1$, the posterior probability of the true model concentrates to 1 at an exponential rate determined by $D_{12}$. Specifically, the error term behaves as $\pi(S^{1*} \mid \mixdata) \approx 1 - e^{-N D_{12}}$, where $N = n_N + m_N$ is the total sample size.

We first examine the rate $D_{12}(\eta)$ as a function of the sample size mixing parameter $\eta$. Lemma \ref{lemma:concavity_D_12} shows that the exponent is strictly concave and the condition \eqref{eq:cond_D_12} reveals an interesting phenomenon. If $y^2 > \tstwo$, the inequality does not hold (the LHS becomes negative) and the exponent is strictly decreasing, suggesting that asymptotically, observational data does not contribute to the optimal rate.

In the case where $y^2 < \tstwo$, the behavior depends on the Signal-to-Noise Ratio (SNR). For low SNR (small $|w^*|$), the linear and logarithmic terms are approximately equal; and hence $D_{12}$ remains strictly decreasing. Similarly, for large SNR (large $|w^*|$), the inequality is violated because the logarithmic term (RHS) grows unbounded while the linear term (LHS) saturates. Therefore, optimal mixing occurs only for intermediate SNR with $y^2 < \tstwo$. In this case, the posterior uses the observational data to better \textit{learn} $w$ and hence a combination of interventional and observational data yields a faster concentration rate.

Figure \ref{fig:comparison_rates} (left) illustrates the behavior of $D_{12}$.  This confirms the insight that in the limit $\eta \to 1$ (purely observational data), the problem becomes non-identifiable, and the exponential driving force disappears. Moreover, this figure shows that two choices of the hyper-parameters which give different behaviors of the rate speed as a function of $\eta$; for the blue line, mixing can improve the speed, while for the red dotted line, the observational samples affect the concentration speed.

Next, we analyze the rate $D_{21}(\eta)$ which governs convergence when the true model is the reverse causal structure $S^2$. In Lemma \ref{lemma:D_21_concavity}, we show that the exponent is strictly concave in $\eta$ and always achieves a maximum inside $(0,1)$. Unlike $D_{12}(\eta)$, Figure \ref{fig:comparison_rates}(right) demonstrates that $D_{21}(\eta)$ vanishes in both extremes: as $\eta \to 1$ and as $\eta \to 0$. 
\begin{itemize}
    \item As $\eta \to 1$ (observational limit), $S^1$ and $S^2$ become observationally equivalent, rendering $D_{21}(\eta) \to 0$.

    \begin{figure*}[!t] 
    \centering
    \includegraphics[width=0.95\textwidth]{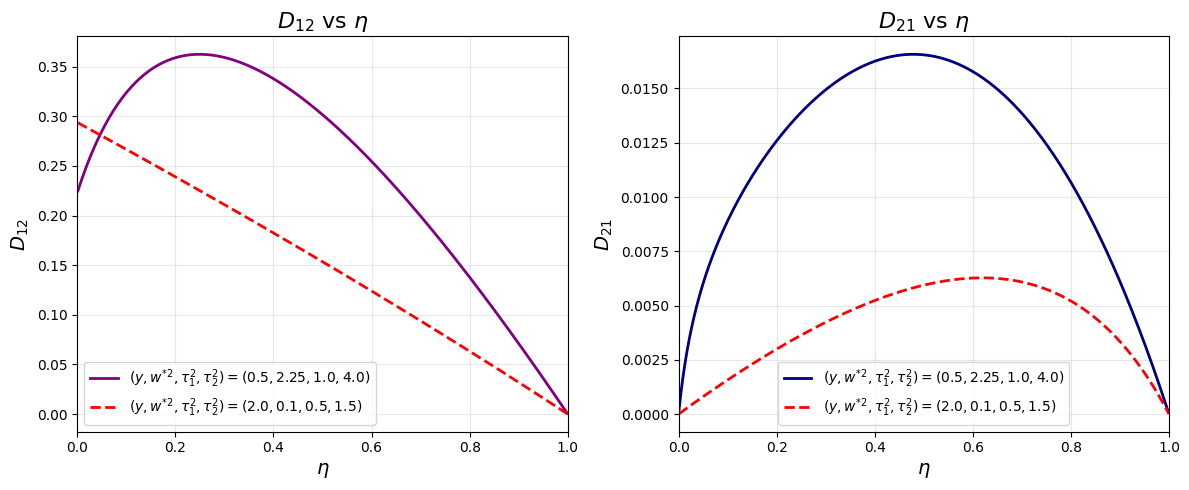}
    \caption{Comparison of the rates $D_{12}$ (left) and $D_{21}$ (right) as a function of $\eta$ for different choices of $[w^*, \tau_1^{*2}, \tau_2^{*2}]^\transpose$.}
    \label{fig:comparison_rates}
\end{figure*}
    \item As $\eta \to 0$ (interventional limit), the intervention $do(Y(2)=y)$ renders $Y(1)$ independent of $y$ in both $S^2$ and $S^3$. Since both models predict $Y(1) \sim N(0, \tsone)$, they become indistinguishable, and the rate again drops to zero.
\end{itemize}

This highlights an important trade-off: observational data is required to distinguish $S^2$ from the independence model $S^3$, while interventional data is required to distinguish $S^2$ from $S^1$. 
\begin{remark}
    Alternatively, parameterizing the posterior error by a fixed interventional budget $m$ yields the approximation $1 - \pi(S^* \mid \mixdata) \approx e^{-m \frac{D(\eta)}{1-\eta}}$. The term $\frac{D(\eta)}{1-\eta}$ characterizes the marginal gain provided by observational data relative to the interventional baseline. As shown in Figure~\ref{fig:new_exponents}, increasing $\eta$ consistently improves the concentration speed. Notably, this improvement varies from incremental gains (e.g., the red dotted line) to significant accelerations, highlighting the important role of observational samples.
\end{remark}

\subsection{Application to BGe - type  Priors}
\label{subsec:bge_models}
We consider the following hierarchical model
\begin{equation}
\label{eq:BGe_priors_compact}
\begin{split}
    S \sim \text{Unif}\{S^1, S^2, S^3\}; \quad \tau_j^2 \mid S^i \sim \text{IG}(\alpha_{i,j}, \beta), \\
    w \mid (S^i, \tau^2) \sim N(0, \lambda\tau_i^2)^{\mathbb{I}(i \neq 3)}
\end{split}
\end{equation}
where $\mathbb{I}$ denotes the indicator function, $\beta, \lambda, \alpha_{i,j} \in (0, \infty)$ for all $i \in \{1, 2, 3\}$, $j \in \{1,2\}$ and $\text{IG}(\alpha, \beta)$ is the Inverse-Gamma distribution with parameters $\alpha$ and $\beta$. We first investigate the case where only observational data is available.

\subsubsection{Observational Data}
For this choice of priors, the posterior distribution over the model structures can be calculated in closed form. The exact formula is given in Appendix \ref{thm:BGe_nonindent}.

\begin{remark}
    Under the specific hyper-parameter choice $\alpha_1 = \alpha_4 = \alpha$, $\alpha_2 = \alpha_3 = \alpha - \frac{1}{2}$, and $2\beta = \lambda$, we have $S_1^{X\lambda} = S_1^{X\beta}$ and $S_2^{X\lambda} = S_2^{X\beta}$, where these quantities are defined in \eqref{eq:def_S_beta_eta}. Consequently, the posterior distribution, $\left[\pi(S^1\mid\obsdata), \pi(S^2\mid\obsdata), \pi(S^3\mid\obsdata)\right]^\transpose$ simplifies to
\begin{equation}
    \frac{1}{c} \left[ \Delta^{-\nu/2}, \Delta^{-\nu/2}, (S_1^{X\beta})^{\frac{1-\nu}{2}}(S_2^{X\beta})^{-\frac{\nu}{2}} \right]^\transpose,
\end{equation}
where $\Delta = S_1^{X\beta} S_2^{X\beta} - (S_{12}^{X})^2$.
where $c$ is a normalizing constant and $\nu = 2\alpha + n$. This choice of priors places equal mass on models within the same Markov equivalence class. \cite{geiger1998characterization} demonstrates that these are the unique priors satisfying this property.
\end{remark}

Empirically, Figure \ref{fig:comparison_post_odds} in the Appendix demonstrates that when $S^1$ is the true mechanism, the posterior odds between the non-identifiable models $S^1$ and $S^2$ do not concentrate at 1, but rather at a constant determined by the priors. Furthermore, we note that with this choice of hyper-parameters the Bayes factor concentrates at a value less than 1, suggesting that the Bayesian approach favors the incorrect model $S^2$ in this scenario. 

Next, to empirically verify the convergence behavior of the independence model $S^3$, we plot the posterior probability of the true structure in Figure~\ref{fig:m3_nonident_conv_speed}. We observe that the posterior mass concentrates at $1$ at a rate of $O_p(n^{-1/2})$, as established in Theorem~\ref{thm:non_ident_S3}. The minor fluctuations observed at larger $n$ are consistent with this stochastic bound; unlike the connected cases which exhibit almost sure  concentration, the independence case is characterized by this slower, probabilistic rate. 

\begin{figure}[h!]
    \centering
    \includegraphics[width=\linewidth]{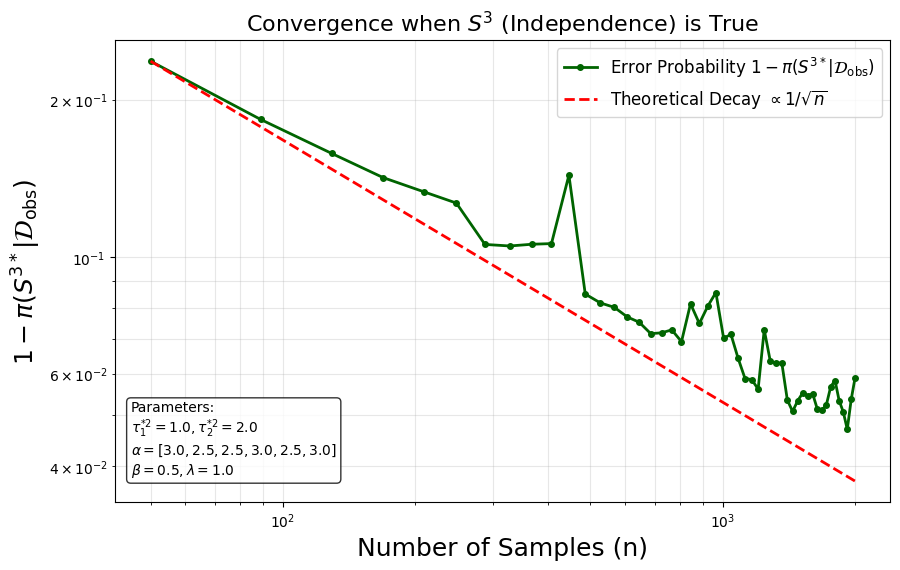}
    \caption{Concentration speed (in log scale) of the posterior $\pi(S^{3*} \mid \obsdata)$ to 1 when the true generating model is $S^3$.}
    \label{fig:m3_nonident_conv_speed}
\end{figure}

To verify the distributional convergence of the log-posterior ratio to a $\chi_1^2$ random variable (see Remark~\ref{remark:log_posterior_chi}), we evaluate the \textit{augmented posterior odds} defined in \eqref{eq:conv_to_chi_sq}. The experimental results presented in Figure~\ref{fig:Bayes_factors_nonident} empirically confirm that these statistics converge in distribution to $\chi_1^2$, aligning with our theoretical predictions.

\begin{figure}[h!]
    \centering
    \includegraphics[width=\linewidth]{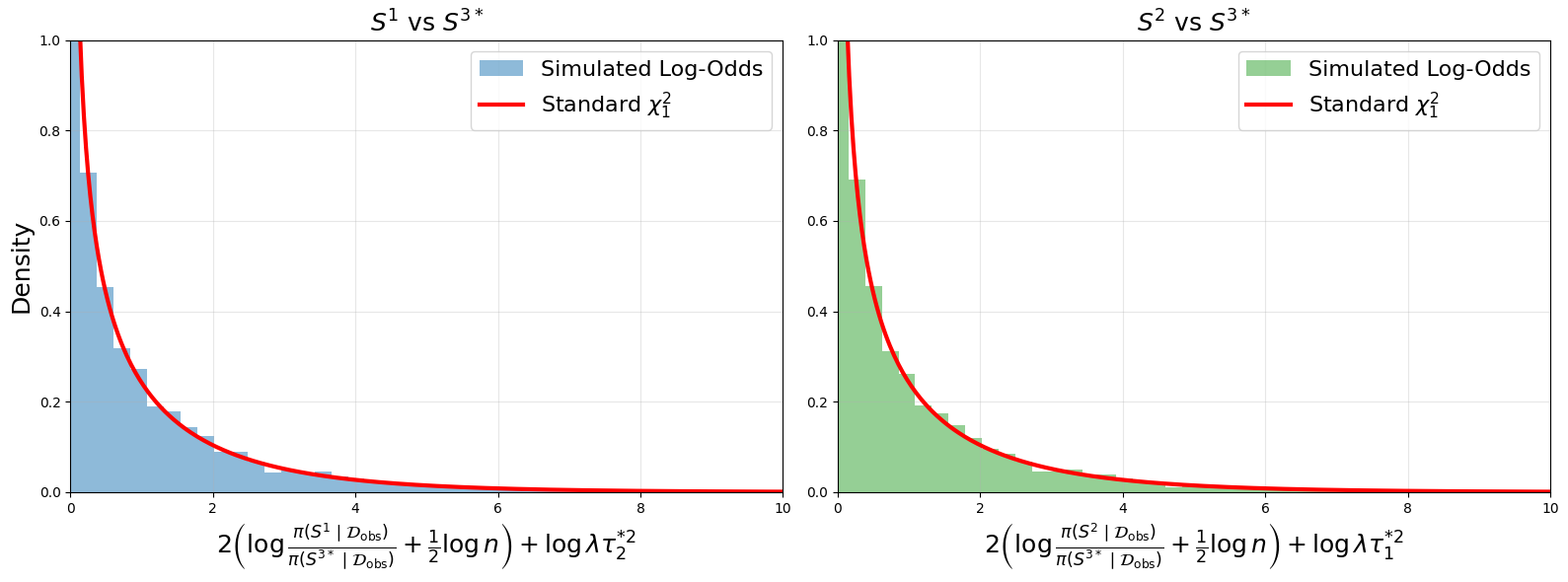}
    \caption{Augmented posterior odds showing convergence to a $\chi_1^2$ distribution when the true generating model is $S^{3*}$.}
    \label{fig:Bayes_factors_nonident}
\end{figure}

\subsubsection{Observational and Interventional Data}
We now add $m$ interventional datapoints which come from the distribution of $\cdot\mid \{\text{do}(Y(2) = y), S\}$. Again, the posterior over the structures has a closed form solution and the exact calculations can be found in Theorem \ref{thm:BGe_interv}

Figure \ref{fig:different_etas_D12_D21} illustrates the posterior's exponential concentration to 1 when the data is generated by either model $S^1$ (left) or $S^2$ (right). For all experiments we used the BGe priors, i.e. $\alpha_1 = \alpha_4 = 3.0$ and $\alpha_2 = \alpha_3 =2.5$ with $2\beta = \lambda = 1.0$, which places equal mass on models $S^1$ and $S^2$ when only observational samples are available. As expected, the closeness to the theoretical prediction varies with the mixing parameter $\eta$. For structure $S^1$, this correspondence degrades as $\eta$ increases from $0.1$ to $0.9$. This behavior is strongly connected to the results in Section \ref{subsec:numerical_analysis_conv_rates}, where the exponent $D_{12}$ approaches zero as the proportion of observational samples increases, effectively nullifying the exponential decrease. One can see that for the hyperparameters used, the theoretical exponent decreases from $0.53$ for $\eta = 0.1$ to $0.30$ for $\eta = 0.5$ and finally to $0.07$ for $\eta = 0.9$. In contrast, for structure $S^2$, the empirical results track the predicted exponential decay most accurately at intermediate values of $\eta$, with performance degrading at the extremes, again in line with  the results of Section \ref{subsec:numerical_analysis_conv_rates}.

Finally, we show the posterior concentration speed to $1$ for $S^3$ in Figure \ref{fig:m3_ident_conv_speed}. As previously observed, the concentration speed follows the $\sqrt{N}$ line, with the fluctuations around given by the stochastic bound.

\begin{figure}[h!]
    \centering
    \includegraphics[width=\linewidth]{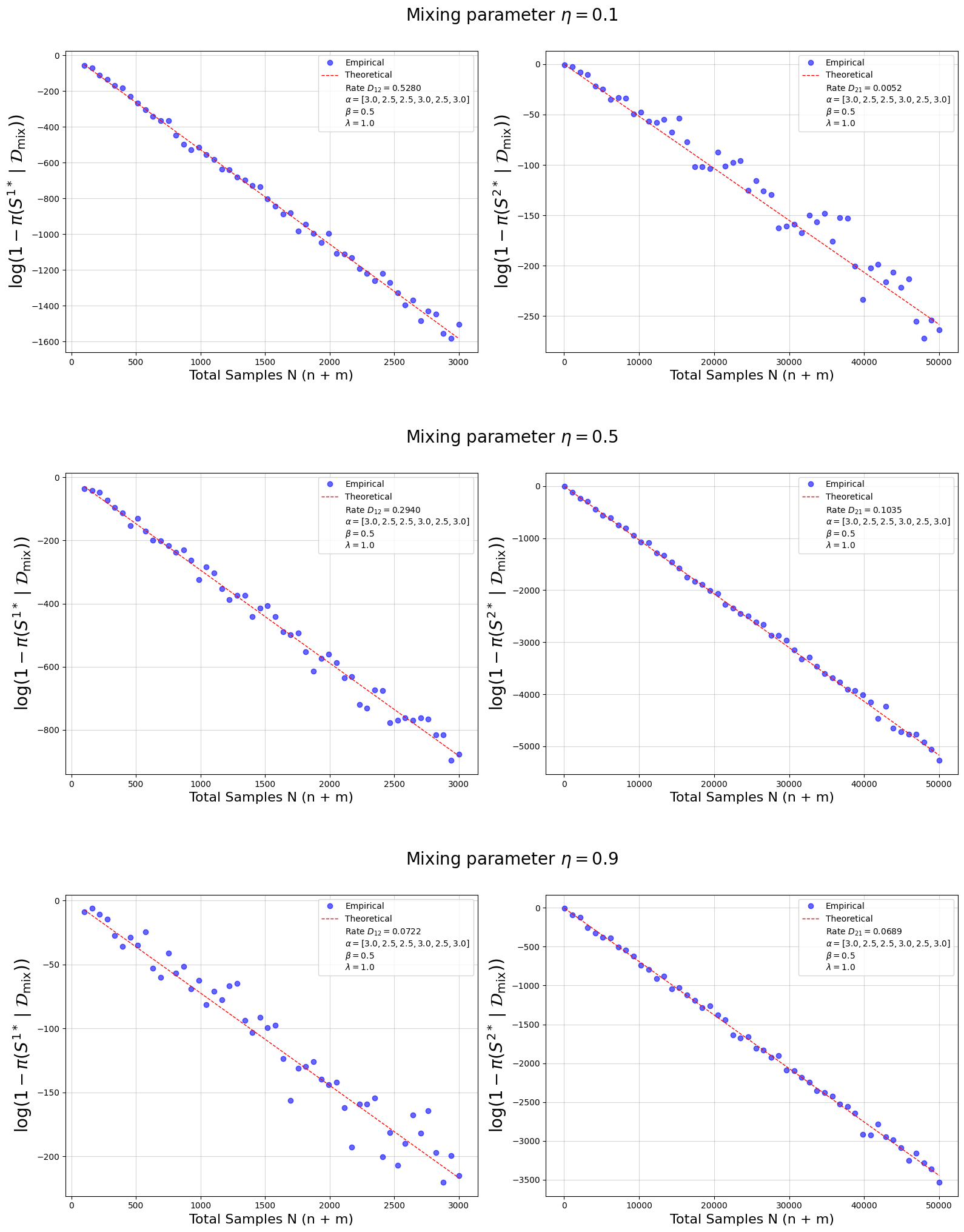}
    \caption{Concentration speed of the posterior to 1 when the true model is $S^1$ (left) or $S^2$ (right).}
    \label{fig:different_etas_D12_D21}
\end{figure}

\begin{figure}[h!]
    \centering
    \includegraphics[width=\linewidth]{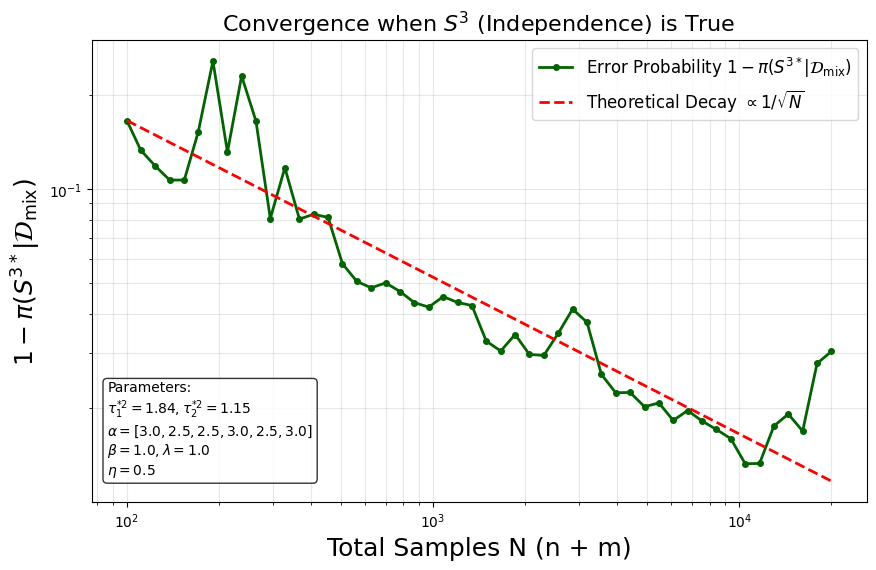}
    \caption{Concentration speed (in log scale) of the posterior $\pi(S^{3*} \mid \mixdata)$ to 1 when the true generating model is $S^3$.}
    \label{fig:m3_ident_conv_speed}
\end{figure}

\section{Conclusions}
We presented the first formal characterization of Bayesian posterior concentration rates for bivariate Gaussian Linear SEMs. With purely observational data, the posterior probability of a connected true structure fails to concentrate to $1$, instead converging to a fixed constant determined by the prior specification. For the independence model, the posterior identifies the structure at a stochastic rate of $O_p(n^{-1/2})$.

Our analysis of the mixed-data regime revealed a sharp concentration dichotomy: the introduction of interventional samples triggers a phase transition from polynomial to exponentially fast convergence for connected graphs. We derived explicit formulae for these concentration exponents and identified  conditions on the observational-to-total sample ratio, $\eta$, where a blend of data sources gives an increase in the  concentration speed. We validated our theoretical results in BGe-type priors. 

While we focused on the linear case, the identified concentration dichotomy suggests that the interplay between observational calibration and interventional signal is a fundamental feature of causal discovery. Extending these asymptotic bounds to non-linear mechanisms and high-dimensional settings remains a promising avenue for future research.


\bibliography{uai2026-template}

\newpage

\onecolumn

\title{Bayesian Causal Discovery: Posterior Concentration Rates with Observational and Interventional Data\\(Supplementary Material)}
\maketitle
\appendix
\section{Some Lemmata}
\subsection{Proof of Lemma \ref{lemma:non_ident}}
\label{appendix:lemma_non_ident}
Rearraging \eqref{eq:SCM_vector}, we obtain $X = (I-A^\transpose)^{-1}\epsilon$, and we note that the inverse of $(I-A^\transpose)$ always exists. Then, under structure $S^1$, we have

    $$\{X\mid \theta^1, S^1\}  \sim N_2\left( \begin{bmatrix}
        0\\ 0
    \end{bmatrix}, \begin{bmatrix}
        w^2\tau_2^2+\tau_1^2 &w\tau_2^2\\
        w\tau_2^2 & \tau_2^2
    \end{bmatrix} \right), \text{ for any } \theta^1 := [w, \tau_1^2, \tau_2^2]^\transpose \in \Theta.$$

    It is easy to check that if the parameters in the second structure, $S^2$, are chosen as follows $\theta^2 = \left[\frac{w\tau_2^2}{w^2\tau_2^2+\tau_1^2},w^2\tau_2^2+\tau_1^2, \frac{\tau_1^2\tau_2^2}{w^2\tau_2^2+\tau_1^2}\right]^\transpose \in \Theta$, then the distribution of $X$ is 

    $$\{X\mid S^2, \theta_2\} \sim N_2\left( \begin{bmatrix}
        0\\ 0
    \end{bmatrix}, \begin{bmatrix}
        w^2\tau_2^2+\tau_1^2 &w\tau_2^2\\
        w\tau_2^2 & \tau_2^2
    \end{bmatrix} \right),$$
    and therefore, the two structures are non-identifiable. 

\begin{lemma}
    \label{lemma:MLE_nonident}
    If $\obsdata := (X_i)_{1\leq i \leq n}$ are $n$ i.i.d. observations generated by \eqref{eq:SCM_vector}, then the MLEs for each structure are

    \begin{align}
        &\hat{\theta}^1 := [\hat{\theta} \mid S^1] = \left[\frac{S^X_{12}}{S^X_2}, \frac{S_1^X S_2^X-(S^{X}_{12})^2}{nS^X_2}, \frac{S_2^X}{n}\right]^\transpose, \\
        &\hat{\theta}^2 := [\hat{\theta} \mid S^2] = \left[\frac{S_{12}^X}{S^X_1}, \frac{S^X_1}{n}, \frac{S_1^XS^X_2-(S_{12}^X)^2}{nS_1^X}\right]^\transpose, \\
        &\hat{\theta}^3 := [\hat{\theta} \mid S^3] = \left[0, \frac{S^X_1}{n}, \frac{S^X_2}{n}\right]^\transpose,
    \end{align}
    where
    \begin{equation}
        S^X_{12} := \sum_{i=1}^n X_i(1)X_i(2), \quad S_1^X:= \sum_{i=1}^n X_i(1)^2, \quad S^X_2:= \sum_{i=1}^n X_i(2)^2.
        \label{eq:def_S12_S1_S2}
    \end{equation}
    \label{lemma:log_lik_simple}
\end{lemma}

\begin{proof}
Let $\ell_n(\theta\mid S^1) := \sum_{i=1}^n\log f(X_i \mid S^1, \theta) $ denote the log-likelihood function for structure $S^1$. By factorizing the joint density as $f(x\mid S^1, \theta) = f(x(1), x(2)\mid S^1, \theta) = f(x(1)\mid x(2), w, \tau_1^2)f(x(2)\mid \tau_2^2)$, and letting $\phi_v(\cdot)$ denote the density of a centered Gaussian with variance $v$, we obtain
\begin{align*}
    \ell_n(\theta\mid S^1) &= \sum_{i=1}^n \left[ \log \phi_{\tau_1^2}(X_i(1)-wX_i(2)) + \log \phi_{\tau_2^2}(X_i(2)) \right] \\
    &= -\frac{n}{2}\log(2\pi \tau_1^2) - \sum_{i=1}^n \frac{(X_i(1)-wX_i(2))^2}{2\tau_1^2} -\frac{n}{2}\log(2\pi \tau_2^2) - \sum_{i=1}^n \frac{X_i(2)^2}{2\tau_2^2} \\
    &= -n\log(2\pi) - \frac{n}{2}\log(\tau_1^2\tau_2^2) - \frac{S_1^X-2wS^X_{12}+w^2S^X_2}{2\tau_1^2} - \frac{S^X_2}{2\tau_2^2}.
\end{align*}
The maximum likelihood estimates are obtained by solving the first-order conditions $\nabla_{\theta} \ell_n(\theta\mid S^1) = 0$. Differentiating with respect to $w$, we find:

\[
\frac{\partial \ell_n(\theta \mid S^1)}{\partial w} = -\frac{1}{2\tau_1^2}(-2S^X_{12} + 2wS^X_2) = 0 \implies \hat{w} = \frac{S^X_{12}}{S^X_2}.
\]

Substituting $\hat{w}$ and differentiating with respect to the variance parameters yields the standard variance estimators. It is easy to compute the second derivatives and show that the optimum is in fact unique. The MLEs for models $S^2$ and $S^3$ follow analogously.
\end{proof}

\begin{lemma}
    \label{lemma:conv_nonident}
    Let $\obsdata := (X_i)_{1\leq i\leq n}$ be $n$ i.i.d. observations generated by \eqref{eq:SCM_vector}. Then, as $n\to \infty$
    \begin{align}
        &\frac{S_1^X}{n} \xrightarrow{a.s.} \mathbb{E}[X(1)^2], \quad \frac{S_2^X}{n} \xrightarrow{a.s.} \mathbb{E}[X(2)^2], \\
        &\frac{S_1^X S_2^X - (S_{12}^X)^2}{n S_2} \xrightarrow{a.s.} \frac{\mathbb{E}[X(1)^2]\mathbb{E}[X(2)^2] - (\mathbb{E}[X(1)X(2)])^2}{\mathbb{E}[X(2)^2]}, \\
        &\frac{S_1^X S_2^X - (S_{12}^X)^2}{n S_1} \xrightarrow{a.s.} \frac{\mathbb{E}[X(1)^2]\mathbb{E}[X(2)^2] - (\mathbb{E}[X(1)X(2)])^2}{\mathbb{E}[X(1)^2]},
    \end{align}
    with $S_1^X$, $S_2^X$ and $S_{12}^X$ defined in \eqref{eq:def_S12_S1_S2}.
\end{lemma}

\begin{proof}
    Rearranging \eqref{eq:SCM_vector}, we have $X = (I - A^\transpose)^{-1}\epsilon$. Since $(I - A^\transpose)$ is invertible, the covariance matrix of $X$ is given by $\text{Var}(X) = (I - A^\transpose)^{-1}\Sigma (I - A^\transpose)^{-\transpose}$. Since the second moments are finite: $\mathbb{E}[X(1)^2] < \infty$ and $\mathbb{E}[X(2)^2] < \infty$, then, by the Cauchy-Schwarz inequality, $\mathbb{E}[|X(1)X(2)|] \leq \sqrt{\mathbb{E}[X(1)^2]\mathbb{E}[X(2)^2]} < \infty$.
    
    Applying the strong law of large numbers to $S^X_1 = \sum X_i(1)^2$, $S^X_2 = \sum X_i(2)^2$, and $S^X_{12} = \sum X_i(1)X_i(2)$, we obtain almost sure (a.s.) convergence to 
    $$ 
    \frac{S_1^X}{n} \xrightarrow{a.s.} \mathbb{E}[X(1)^2], \quad \frac{S_2^X}{n} \xrightarrow{a.s.} \mathbb{E}[X(2)^2], \quad \frac{S_{12}^X}{n} \xrightarrow{a.s.} \mathbb{E}[X(1)X(2)]. 
    $$
    We define the function $g(x, y, z) = \frac{xy - z^2}{y}$. Note that since $X$ follows a non-degenerate Gaussian distribution, $\mathbb{E}[X(1)^2] > 0$ and $\mathbb{E}[X(2)^2] > 0$. Consequently, the vector of sample moments converges almost surely to a point in the domain where $g$ is continuous. Therefore
    $$
    \frac{S_1^X S_2^X - (S_{12}^X)^2}{n S_2} = g\left(\frac{S_1^X}{n}, \frac{S_2^X}{n}, \frac{S_{12}^X}{n}\right) \xrightarrow{a.s.} \frac{\mathbb{E}[X(1)^2]\mathbb{E}[X(2)^2] - (\mathbb{E}[X(1)X(2)])^2}{\mathbb{E}[X(2)^2]}.
    $$
    The result for the second estimator follows analogously.
\end{proof}

\begin{lemma}
    \label{lemma:laplace_regular}
    Let $B_{\delta}(\theta)$ denote the open ball of radius $\delta$ centered at $\theta \in \Theta$. For any $S \in \mathcal{S}$ and $\theta \in \Theta$, let $\obsdata = (X_i)_{1\leq i \leq n}$ be $n$ i.i.d. observations generated by \eqref{eq:SCM_vector} with law $P_{m(S, \theta)}$ and density $f(x \mid S, \theta)$. The following four conditions hold
    
    \begin{enumerate}
    \item[(i)] for all $x_1^n = (x_i)_{1\leq i\leq n}$ with $x_i\in \mathbb{R}^2$  and $\theta \in \Theta$, $f(x_1^n\mid S, \theta)> 0$. Moroever, for all $x_1^n$, $\ell_n(\theta\mid S) = \sum_{i=1}^n  \log f(x_i\mid S, \theta)$ is six times continuously differentiable in $\theta$;
    \item[(ii)] for any $\theta_0 \in \Theta$ there exist $\epsilon > 0$ and $M < \infty$ such that $B_\epsilon(\theta_0) \subseteq \Theta$ and for $1 \le j_1, \dots, j_d \le m$ with $d \le 6$,
    \[
    \limsup_{n \to \infty}  \sup \{ n^{-1} | \partial_{j_1} \dots \partial_{j_d} \ell_n(\theta\mid S) | : \theta \in B_\epsilon(\theta_0) \} < M,
    \]
    with $P_{m(S,\theta_0)}$-probability one;
    \item[(iii)] for any $\theta_0 \in \Theta$ and $\epsilon > 0$,
    \[
    \limsup_{n \to \infty}  \sup \{ n^{-1} \det(\nabla^2_\theta \ell_n(\theta\mid S)) : \theta \in B_\epsilon(\theta_0) \} < 0
    \]
    with $P_{m(S, \theta_0)}$-probability one;
    \item[(iv)] for any $\theta_0 \in \Theta$ and any $\delta > 0$,
    \[
    \limsup_{n \to \infty} \sup \{ n^{-1} [\ell_n(\theta\mid S) - \ell_n(\theta_0\mid S)]: \theta \in \Theta \setminus B_\delta(\theta_0) \} < 0
    \]
    with $P_{m(S,\theta_0)}$-probability one.
\end{enumerate}
Hence, the models are Laplace regular as per \citet{Kass1990TheVO}.
\end{lemma}

\begin{proof}
    We check the four conditions presented in \citet{Kass1990TheVO} which allow us to approximate the marginal likelihoods using Laplace's method. Without loss of generality we assume $S = S^1$. The other  remaining cases follow along the same lines. Under $S^1$, the log-likelihood $\ell_n(\theta \mid  S^1)$ as a function of the MLE $\hat{\theta}^1$ is

    \begin{align}
    \label{eq:MLE_expansion_S1}
    \ell_n(\theta \mid  S^1) &= -n \log 2\pi -\frac{n}{2}\log \tau_1^2\tau_2^2 - \frac{n}{2}\cdot\frac{(w-\hat{w})^2\hat{\tau}^2_2}{\tau_1^2} - \frac{n}{2}\cdot \frac{\hat{\tau}_1^2}{\tau_1^2}- \frac{n}{2}\cdot \frac{\hat{\tau}_2^2}{\tau_2^2}.
    \end{align}

    \begin{itemize}
        \item[(i)] We note that $\ell_n( \theta \mid S^1)$ is 6-times continuously differentiable and $f(x_1^n \mid S, \theta) >0$ for all $\obsdata$ and $\theta \in \Theta$. 

        \item[(ii)] The log-likelihood function $\ell_n(\theta\mid S_1)$ decomposes into a sum of terms that are either quadratic polynomials in the structural parameter $w$ or rational functions of the variance parameters $\tau_1^2, \tau_2^2$.

        Specifically, partial derivatives with respect to $w$ of order $d \ge 3$ vanish and hence are bounded. Partial derivatives with respect to the variances involve terms of the form $\tau^{-k}$, which are $C^\infty$ smooth and bounded on any compact domain where the variances are bounded away from zero (i.e., $\tau_j^2 \ge \delta > 0$). Consequently, all mixed partial derivatives of order $d \ge 3$, when normalized by $n^{-1}$, are uniformly bounded in a neighborhood of the true parameter, satisfying the necessary condition.

        \item[(iii)] The Hessian is 

            \begin{equation}
                H(\theta):= H(w, \tau_1^2, \tau_2^2) = 
                \begin{bmatrix} 
                -\frac{n \hat{\tau}_2^2}{\tau_1^2} & \frac{n \hat{\tau}_2^2 (w - \hat{w})}{(\tau_1^2)^2} & 0 \\
                \frac{n \hat{\tau}_2^2 (w - \hat{w})}{(\tau_1^2)^2} & \frac{n}{2(\tau_1^2)^2} - \frac{n}{(\tau_1^2)^3} \left[ \hat{\tau}_1^2 + \hat{\tau}_2^2(w-\hat{w})^2 \right] & 0 \\
                0 & 0 & \frac{n}{2(\tau_2^2)^2} - \frac{n \hat{\tau}_2^2}{(\tau_2^2)^3}
                \end{bmatrix}
            \end{equation}

            with determinant

            $$\det(H(\theta)) = -\frac{n^3\hat{\tau}_2^2}{4\tau_1^6\tau_2^4}\Big(1-\frac{2\hat{\tau}_2^2}{\tau_2^2}\Big)\Big(1-\frac{2\hat{\tau}_1^2}{\tau_1^2}\Big).$$
            
            For simplicity we define the normalized determinant function $h(\theta) = n^{-3} \det(H(\theta))$. The determinant evaluated at the MLE is $\det(H(\hat{\theta})) = -n^3 / (4 \hat{\tau}_1^6 \hat{\tau}_2^4)$, which implies that $h(\hat{\theta}) = -1/(4 \hat{\tau}_1^6 \hat{\tau}_2^4) := -C < 0$, where $C$ is a positive constant with probability 1. 
            
            The components of the Hessian matrix are linear combinations of rational functions of the variance parameters and polynomial functions of $w$. Since the variances are strictly positive in the parameter space $\Theta$, the determinant is a continuous function of $\theta$.  Moreover, $\hat{\tau}_1^2$ and $\hat{\tau}_2^2>0$ a.s. By the definition of continuity, for any $\delta > 0$, there exists an $\epsilon > 0$ such that $|h(\theta) - h(\hat{\theta})| < \delta$ for all $\theta \in B_\epsilon(\hat{\theta})$. Choosing $\delta = C/2$, we obtain $h(\theta) < -C + C/2 = -C/2 < 0$ for all $\theta$ in this neighborhood. This strictly negative upper bound ensures that the condition $\limsup_{n \to \infty} \sup_{\theta \in B_\epsilon(\hat{\theta})} \{ n^{-1} \det(H(\theta)) \} < 0$ is satisfied with probability one.

        \item[(iv)] 

Condition (iv) requires the global consistency of the MLE. For simplicity, we consider only the $S^1$ case; the other cases follow analogously. Let $\hat{\theta}^1 = [\hat{w}, \hat{\tau}_1^2, \hat{\tau}_2^2]^\top$ denote the MLE under model $S^1$. Using the expansion in \eqref{eq:MLE_expansion_S1}, the normalized difference between the log-likelihood at a parameter $\theta = [w, \tau_1^2, \tau_2^2]^\top$ and at the MLE is given by

\begin{equation*}
g_n(\theta) := \frac{1}{n}\left[\ell_n(\theta \mid S^1) - \ell_n(\hat{\theta}^1 \mid S^1) \right] = -\frac{1}{2} \left( \frac{(w-\hat{w})^2\hat{\tau}^2_2}{\tau_1^2} + \phi\left(\frac{\hat{\tau}_1^2}{\tau_1^2}\right) + \phi\left(\frac{\hat{\tau}_2^2}{\tau_2^2}\right) \right),
\end{equation*}
where we define $\phi(u) = u - 1 - \log u$ for $u > 0$. It is a standard result that $\phi(u) \geq 0$ for all $u > 0$, with equality holding if and only if $u = 1$. Furthermore, $\phi(u)$ is strictly decreasing on $(0, 1)$, strictly increasing on $(1, \infty)$, and satisfies $\lim_{u \to 0}\phi(u) = \lim_{u\to \infty} \phi(u) = \infty$. Because each term inside the bracket is non-negative, $g_n(\theta) \leq 0$ globally, with equality strictly at $\theta = \hat{\theta}^1$.

The primary technical challenge is that the parameter space $\Theta = \mathbb{R}\setminus\{0\}\times (0, \infty)^2$ is not compact. We therefore proceed by partitioning $\Theta \setminus B_\delta(\theta_0)$ into a non-compact region and a compact region.

By Lemma \ref{lemma:conv_nonident}, the MLE is consistent, i.e. $\hat{\theta}^1 \xrightarrow{a.s.} \theta_0$. Consequently, the sequence $\hat{\theta}^1$ is eventually bounded almost surely. There exist positive constants $c, C$ such that for $n$ sufficiently large, we have $\hat{\tau}_1^2, \hat{\tau}_2^2 \in [c, C]$ and $|\hat{w}| \leq C$.

Fix an arbitrary constant $M > 0$. We construct a compact set $K := [-W, W] \times [l, L] \times [l, L]$ containing $B_\delta(\theta_0)$ by choosing the positive constants $W, l, L$ such that

\begin{align}
    \label{cond1}
    L > C \quad &\text{and}\quad \phi\left(\frac{C}{L} \right) > 2M\\
    \label{cond2}
    l < c \quad &\text{and}\quad  \phi\left(\frac{c}{l} \right) > 2M\\
    \label{cond3}
    W > C &+ \sqrt{\frac{2ML}{c}}
\end{align}

Conditions \eqref{cond1} and \eqref{cond2} are trivially satisfied by taking $L$ sufficiently large and $l$ sufficiently small, due to the behavior of $\phi$ when  $u \to 0$ and $u \to \infty$. Next, we demonstrate that for any $n$ large enough and any $\theta \in \Theta \setminus K$, we have $g_n(\theta) < -M$. We evaluate the following three cases

\begin{itemize}
    \item \textbf{Case 1:} If $\tau_i^2 > L$ for at least one $i \in \{1,2\}$, then $\frac{\hat{\tau}_i^2}{\tau_i^2} < \frac{C}{L} < 1$. Because $\phi$ is strictly decreasing on $(0, 1)$, and the condition in \eqref{cond1} implies $\phi\left( \frac{\hat{\tau}_i^2}{\tau_i^2} \right) > \phi\left(\frac{C}{L}\right) > 2M$. This ensures $g_n(\theta) < -M$.
    \item \textbf{Case 2:} If $\tau_i^2 < l$ for at least one $i \in \{1,2\}$, then $\frac{\hat{\tau}_i^2}{\tau_i^2} > \frac{c}{l} > 1$. Because $\phi$ is strictly increasing on $(1, \infty)$, and the condition in \eqref{cond2} implies $\phi\left( \frac{\hat{\tau}_i^2}{\tau_i^2} \right) > \phi\left(\frac{c}{l}\right) > 2M$. This ensures $g_n(\theta) < -M$.
    \item \textbf{Case 3:} Assume $\theta$ is not caught by Cases 1 or 2, which guarantees $\tau_1^2 \leq L$. Because $\theta \notin K$, it must be that $|w| > W$. We then have $\frac{(w-\hat{w})^2\hat{\tau}_2^2}{\tau_1^2} \geq \frac{(|w|-|\hat{w}|)^2 c}{L} \geq \frac{(W-C)^2 c}{L}$. By the condition in \eqref{cond3}, this term is strictly greater than $2M$, ensuring $g_n(\theta) < -M$.
\end{itemize}

Thus, we have established  that $\sup_{\theta \in \Theta \setminus K} g_n(\theta) \leq -M$ eventually almost surely.

We now restrict our attention to the compact set, $K \setminus B_\delta(\theta_0)$, bounded away from the true parameter. Because $\hat{\theta}^1 \xrightarrow{a.s.} \theta_0$ and $g_n(\theta)$ is continuous in both $\theta$ and $\hat{\theta}^1$, the continuous mapping theorem yields pointwise convergence
\begin{equation*}
g_n(\theta) \xrightarrow{a.s.} g_{\infty}(\theta) = -\frac{1}{2} \left( \frac{(w-w_0)^2{\tau}^2_{2,0}}{\tau_{1}^2} + \phi\left(\frac{\tau_{1, 0}^2}{\tau_1^2}\right) + \phi\left(\frac{\tau_{2, 0}^2}{\tau_2^2}\right) \right).
\end{equation*}

Moreover, since $K \setminus B_{\delta}(\theta_0)$ is compact, the above convergence implies uniform converge as well, i.e.

\begin{equation*}
\sup_{\theta \in K \setminus B_\delta(\theta_0)} |g_n(\theta) - g_\infty(\theta)| \xrightarrow{a.s.} 0.
\end{equation*}

Notice that $g_\infty(\theta) \leq 0$ everywhere, with $g_\infty(\theta) = 0$ if and only if $\theta = \theta_0$. Because $\theta_0 \notin K \setminus B_\delta(\theta_0)$ and the set is compact, $g_\infty(\theta)$ must attain a strictly negative maximum on this set. Let $\sup_{\theta \in K \setminus B_\delta(\theta_0)} g_\infty(\theta) = -M'$ with $M' > 0$. Therefore

\begin{align*}
    \sup_{\theta \in K \setminus B_\delta(\theta_0)} g_n(\theta) &= \sup_{\theta \in K \setminus B_\delta(\theta_0)}  (g_n(\theta) -g_\infty(\theta)+g_\infty(\theta))\\
    &\leq \sup_{\theta \in K \setminus B_\delta(\theta_0)}  g_\infty(\theta) + |g_n(\theta)- g_\infty(\theta)|\\
    &\leq \sup_{\theta \in K \setminus B_\delta(\theta_0)}  g_\infty(\theta) + \sup_{\theta \in K \setminus B_\delta(\theta_0)}|g_n(\theta)- g_\infty(\theta)| \leq -M' + \frac{M'}{2} = -\frac{M'}{2},
\end{align*}
eventually almost surely. Combining the bounds from the two regions, we bound the supremum over the entire set $\Theta \setminus B_\delta(\theta_0)$
\begin{equation*}
\sup_{\theta \in \Theta \setminus B_\delta(\theta_0)} g_n(\theta) = \max\left\{\sup_{\theta \in K \setminus B_\delta(\theta_0)} g_n(\theta), \sup_{\theta \in \Theta \setminus K} g_n(\theta) \right\} \leq \max \left\{ -M'/2, -M \right\} < 0.
\end{equation*}

Taking the limit supremum yields
\begin{equation*}
\limsup_{n \to \infty}\sup_{\theta \in \Theta \setminus B_\delta(\theta_0)} g_n(\theta) < 0 \quad \text{a.s.}
\end{equation*}

Finally, observe that $\frac{1}{n} \left[ \ell_n(\theta\mid S^1) - \ell_n(\theta_0\mid S^1) \right] = g_n(\theta) - g_n(\theta_0)$. Because $g_n(\theta_0) \xrightarrow{a.s.} g_{\infty}(\theta_0) = 0$, the limit supremum of the difference is identical to the limit supremum of $g_n(\theta)$. Therefore
\begin{equation*}
\limsup_{n \to \infty}\sup_{\theta \in \Theta \setminus B_\delta(\theta_0)} \frac{1}{n} \left[ \ell_n(\theta\mid S^1) - \ell_n(\theta_0\mid S^1) \right] < 0,
\end{equation*}
with $P_{m(S,\theta_0)}$-probability one, completing the proof.
\end{itemize}

Since all these conditions are fulfilled, then by Theorem 7 in \citet{Kass1990TheVO} the sequence of log-likelihood functions is Laplace regular with $P_{m(S,\theta_0)}-$probability one, for any $\theta_0\in \Theta$.
\end{proof}

\begin{lemma}
    \label{lemma:MLEs_ident}
    Let $\obsdata = (X_i)_{1\leq i \leq n}$ be $n$ i.i.d. samples generated by \eqref{eq:SCM_vector} and $\intdata = (Y_i)_{1\leq i \leq m}$ be $m$ i.i.d. samples generated by the interventional model. The maximum likelihood estimators for each model are

    \begin{align*}
        \hat{\theta}^1:= [\hat{\theta} \mid S^1] &= \left[ \frac{S^X_{12}+S_{12}^Y}{S_2^X+S_2^Y},  \frac{(S_2^X+S_2^Y)(S_1^X+S_1^Y)-(S_{12}^X+S_{12}^Y)^2}{(n+m)(S_2^Y+S_2^X)}, \frac{S_2^X}{n}\right]^\transpose,\\
        \hat{\theta}^2:= [\hat{\theta} \mid S^2]  &= \left[\frac{S_{12}^X}{S_1^X}, \frac{S_{1}^X+S_{1}^Y}{n+m}, \frac{S_1^XS_2^X-S_{12}^{X2}}{nS_1^X} \right]^\transpose,\\
        \hat{\theta}^3:= [\hat{\theta} \mid S^3] &= \left[ 0, \frac{S_1^X+S_1^Y}{n+m}, \frac{S_2^X}{n}\right]^\transpose,\\
    \end{align*}
    where
    \begin{equation}
        S^X_{12} := \sum_{i=1}^n X_i(1)X_i(2), \quad S^X_1:= \sum_{i=1}^n X_i(1)^2, \quad S^X_2:= \sum_{i=1}^n X_i(2)^2,
        \label{eq:MLE_interv_X}
    \end{equation}
    \begin{equation}
        S^Y_{12} := y\sum_{i=1}^m Y_i(1), \quad S^Y_1:= \sum_{i=1}^n Y_i(1)^2, \quad S^Y_2:= my^2.
        \label{eq:MLE_interv_Y}
    \end{equation}
\end{lemma}
\begin{proof}
First, we consider the $S^1$ model. The log-likelihood for this case is

\begin{align*}
    \ell_{n,m}(\theta \mid S^1) :&= \log f(\mixdata\mid \theta, S^1) \\
    &= \sum_{i=1}^n \Big[ \log\phi_{\tau_1^2}(X_i(1)-wX_i(2))^2 + \log \phi_{\tau_2^2}(X_i(2)) \Big] + \sum_{j=1}^m \log_{\tau_1^2}(Y_i(1) - wy) \\
    & = -\Big(n+\frac{m}{2}\Big) \log 2\pi -\frac{n+m}{2} \log \tau_1^2 -\frac{n}{2}\log \tau_2^2 -\frac{(S_1^X+S_1^Y)-2w(S_{12}^X+S_{12}^Y) + w^2(S_2^Y+S_2^X)}{2\tau_1^2} - \frac{S_2^X}{2\tau_2^2}
\end{align*}

The maximum likelihood estimates are obtained by solving the first-order conditions $\nabla_{\theta} \ell_{n,m}(\theta\mid  S^1) = 0$. Differentiating with respect to $w$, we find:
\[
\frac{\partial \ell_{n,m}(\theta\mid  S^1)}{\partial w} = -\frac{1}{2\tau_1^2}\Big[-2(S^X_{12}+S_{12}^Y) + 2w(S_2^X+S_2^Y)\Big] = 0 \implies \hat{w} = \frac{S^X_{12}+S_{12}^Y}{S_2^X+S_2^Y}.
\]
Substituting $\hat{w}$ and differentiating with respect to the variance parameters yields the variance estimators. Moreover, one can easily check that $\ell_{n, m}(\theta \mid S^1)$ is strictly concave $\theta$ by showing that the Hessian is negative definite. Therefore, the maximum is unique.  The MLEs for models $S^2$ and $S^3$ follow analogously.
\end{proof}

\subsection{Proof of Lemma \ref{lemma:concavity_D_12}}

\begin{proof}
    The explicit form derived is
\begin{equation*}
    D_{12}(\eta) = \frac{1}{2}\log \left( \eta \sigma_{1,X}^2 + \bar{\eta}\sigma_{1, Y}^2 \right) - \frac{\eta}{2}\log \sigma_{1,X}^2 - \frac{\bar{\eta}}{2}\log \tsone,
\end{equation*}
where $\sigma_{1,X}^2$ and $\sigma_{1, Y}^2$ are defined in \eqref{eq:D_12_def} and $\eta \in (0,1)$. We investigate the analytic properties of this exponent by computing the first two derivatives with respect to $\eta$
\begin{align*}
    \frac{dD_{12}(\eta)}{d\eta} &= \frac{\sigma_{1,X}^2 - \sigma_{1, Y}^2}{2[\eta(\sigma_{1,X}^2-\sigma_{1,Y}^2)+\sigma_{1,Y}^2]} -  \frac{1}{2}\log\sigma_{1,X}^2 + \frac{1}{2}\log \tsone, \\
    \frac{d^2D_{12}(\eta)}{d\eta^2} &= -\frac{(\sigma_{1,X}^2-\sigma_{1,Y}^2)^2}{2(\eta\sigma_{1,X}^2+\bar{\eta}\sigma_{1,Y}^2)^2}.
\end{align*}
Observing the second derivative, we conclude that the function is strictly concave provided $\sigma_{1,X}^2 \neq \sigma_{1,Y}^2$ (or equivalently, $y^2 \neq \tstwo$). Furthermore, we note the boundary conditions: $\lim_{\eta \to 0}D_{12}(\eta) = \frac{1}{2} \log\left(1+\frac{w^{*2}y^2}{\tsone} \right) > 0$ for all $y \neq 0$, and $\lim_{\eta\to1}D_{12}(\eta) = 0$.

Consequently, if the first derivative evaluated at $\eta=0$ is non-positive, the exponent is strictly decreasing in $\eta$, implying that the maximal rate is achieved at $\eta^* = 0$. Conversely, if
\begin{equation}
    \label{eq:first_der_D12}
    \left. \frac{dD_{12}(\eta)}{d\eta}\right|_{\eta = 0} > 0 \iff \frac{w^{*2}(\tstwo-y^2)}{2(w^{*2}y^2+\tsone)} > \frac{1}{2}\log \left( 1+\frac{w^{*2}\tstwo}{\tsone} \right),
\end{equation}
then $D_{12}(\eta)$ achieves a maximum at an interior point $\eta^* \in (0,1)$. 
\end{proof}

\subsection{Proof of Lemmma \ref{lemma:D_21_concavity}}
\begin{proof}
    The rate is given by:
\begin{equation}
    D_{21}(\eta) = \frac{1}{2} \log \left( 1 - \frac{\eta^2 w^{*2}\tsone}{ \eta(w^{*2}\tsone+\tstwo) + \bar{\eta} y^2} \right) + \frac{\eta}{2} \log \left(1+\frac{w^{*2}\tsone}{\tstwo} \right).
\end{equation}
One can easily show that 
$$\frac{d^2}{d\eta^2}\left( \frac{\eta^2w^{*2}\tsone}{\eta(w^{*2}\tsone+\tstwo-y^2)+ y^2}\right) = \frac{2w^{*2}\tsone y^4}{[(w^{*2}\tsone+\tstwo-y^2)\eta+y^2]^3} > 0,$$
which indicates that the inner term in the first logarithm is strictly convex. This, combined with the fact that the second term in $D_{21}$ is linear in $\eta$ (hence concave) and the function $u \mapsto \log(1-u)$ is concave and decreasing for $u\in (0,1)$, shows that $D_{21}(\eta)$ is strictly concave. Moreover, since the function is continuous in $\eta$ and 

$$\lim_{\eta \to 0^+}D_{21}(\eta) = 0 \text{, and } \lim_{\eta \to 1^-}D_{21}(\eta) = 0,$$

then the maximum is achieved for some $\eta^* \in (0,1)$.
\end{proof}

\section{Proof of Theorem \ref{thm:nonident_mec}}
\label{appendix:proof_thm1_nonident}
\begin{proof}
    Using the definitions introduced in Lemma \ref{lemma:log_lik_simple}, the marginal likelihood of the data $\obsdata$ conditioned on model $S^1$ is given by the integral of the likelihood weighted by the prior
    \begin{equation}
        f(\obsdata\mid S^{1*}) = \int_{\Theta} f(\obsdata\mid S^{1*}, \theta)\pi(\theta|S^1) \, d\theta = \int_{\Theta} \exp\left(\ell_n( \theta\mid S^{1*})\right)\pi(\theta\mid S^{1*}) \, d\theta.
    \end{equation}
    Since the models considered satisfy the Laplace regularity conditions (Lemma \ref{lemma:laplace_regular}), the log-marginal likelihood can be asymptotically approximated as $n \to \infty$
    \begin{align*}
        \log f(\obsdata\mid S^{1*}) &= \ell_n(\hat{\theta}^1\mid S^{1*}) - \frac{3}{2}\log n + \frac{3}{2}\log(2\pi) - \frac{1}{2}\log \det(I_1(\hat{\theta}^1)) + \log \pi(\hat{\theta}^1\mid S^{1*}) + O(n^{-1}) \text{ a.s.}
    \end{align*}
    where $I_1(\theta) =  Diag\left(  \frac{\tau_2^2}{\tau_1^2}, \frac{1}{2\tau_1^4}, \frac{1}{2\tau_2^4} \right)$ is the Fisher information matrix under structure $S^1$ (parameterized by $\theta = [w, \tau_1^2, \tau_2^2]^\transpose$).
    Similarly, for the reverse causal model $S^2$ and the independence model $S^3$, we have
    \begin{align}
        \label{eq:Laplace_approx_nonident}
        \noindent \log f(\obsdata\mid S^2) &= \ell_n(\hat{\theta}^2\mid S^2) - \frac{3}{2}\log n - \frac{1}{2}\log \det(I_2(\hat{\theta}^2)) + \log \pi(\hat{\theta}^2\mid S^2) + \frac{3}{2}\log 2\pi + O(n^{-1})\text{ a.s.}, \\
        \log f(\obsdata\mid S^3) &= \ell_n(\hat{\theta}^3\mid S^3) - \frac{2}{2}\log n - \frac{1}{2}\log \det(I_3(\hat{\theta}^3)) + \log \pi(\hat{\theta}^3\mid S^3) + \log 2\pi + O(n^{-1}) \text{ a.s.}
    \end{align}
    The Fisher information matrix for $S^2$ (parameterized by the reverse weight and variances) is $Diag\left( \frac{\tau_1^2}{\tau_2^2}, \frac{1}{2\tau_2^4}, \frac{1}{2\tau_1^4} \right)$. For the independence model $S^3$, the parameter space is reduced to the two variances ($w=0$), yielding the $2 \times 2$ information matrix: $Diag \left(\frac{1}{2\tau_1^4}, \frac{1}{2\tau_2^4} \right)$.
    
    We first analyze the log-marginal likelihood ratio between the observationally equivalent models $S^1$ and $S^2$. Assuming the true generating structure is $S^{1*}$ with parameter $\theta^*$, we note that $\ell_n(\hat{\theta}^1\mid S^1) = \ell_n(\hat{\theta}^2\mid S^2)$ and the dimension penalty terms are equal. Thus, as $n \to \infty$
    \begin{equation}
        \log \frac{f(\obsdata\mid S^{1*})}{f(\obsdata\mid S^2)} = \frac{1}{2}\log \frac{\det(I_2(\hat{\theta}^2))}{\det(I_1(\hat{\theta}^1))} + \log \frac{\pi(\hat{\theta}^1\mid S^{1*})}{\pi(\hat{\theta}^2\mid S^2)} + O(n^{-1}) \text{ a.s}.
    \end{equation}
    
    By the result of Lemma \ref{lemma:conv_nonident}, $\hat{\theta}^1 \xrightarrow{a.s.} \theta^*$, $\hat{\theta}^2 \xrightarrow{a.s.} \gamma(\theta^*)$ and $\hat{\theta}^3 \xrightarrow{a.s.}   [0, w^{*2}\tstwo+\tsone, \tstwo]^\transpose=:\theta^3_\infty $. The Fisher information matrices satisfy the transformation rule $I_1(\theta) = J_\gamma(\theta)^\top I_2(\gamma(\theta)) J_\gamma(\theta)$,  with $J_\gamma(\theta)$ being the Jacobian of the transformation. This implies that $\det(I_1(\theta)) = \det(I_2(\gamma(\theta))) \cdot (\det J_\gamma(\theta))^2$ which gives
    \begin{align}
        \log \frac{f(\obsdata\mid S^{1*})}{f(\obsdata|S^2)} &\xrightarrow{a.s.} \frac{1}{2}\log \left( \frac{\det(I_2(\gamma(\theta^*)))}{\det(I_2(\gamma(\theta^*))) (\det J_\gamma(\theta^*))^2} \right) + \log \frac{\pi(\theta^*\mid S^{1*})}{\pi(\gamma(\theta^*)\mid S^2)} \\
        &= \log \left( \frac{\pi(\theta^*\mid S^{1*})}{\pi(\gamma(\theta^*)\mid S^2) |\det J_\gamma(\theta^*)|} \right) = \log \left( \frac{\pi(\theta^*\mid S^{1*})}{\pi(\gamma(\theta^*)\mid S^2)} \cdot \frac{w^{*2}\tstwo+\tsone}{\tstwo} \right) \text{ a.s.}
    \end{align}
    This is precisely the log-ratio of the prior density, $\pi(\theta^*\mid S^{1*})$ to the push-forward of the alternative prior, $\gamma^{-1}\#\pi(\theta^* \mid S^2)$. For the incorrect model $S^3$, the normalized log marginal likelihood ratio is 

    $$\frac{1}{n}\log\frac{f(\obsdata\mid S^3)}{f(\obsdata\mid S^{1*})} = \frac{1}{n}\left( \ell_n(\hat{\theta}^3\mid S^3) - \ell_n(\hat{\theta}^{1}\mid S^{1*}) \right) +O \left( \frac{\log n}{n}\right) \text{ a.s.}$$
    Since the log-likelihood functions are continuous in $\theta \in \Theta$, by the continuous mapping theorem and the strong law of large numbers, the normalized log-marginal likelihood ratio converges almost surely to the relative entropy between $P_{m(S^{1*}, \theta^*)}$ and $P_{m(S^3, \theta^3_\infty)}$
    
    \begin{equation}
        \frac{1}{n} \log \frac{f(\obsdata\mid S^3)}{f(\obsdata\mid S^{1*})} \xrightarrow{a.s.} -D(P_{m(S^{1*},\theta^*)} \| P_{m(S^3, \theta^3_\infty)}) = -\frac{1}{2}\log \left(1+\frac{w^{*2}\tstwo}{\tsone} \right) < 0,
    \end{equation}
    where $\theta^3_\infty  = [0, w^{*2}\tstwo+\tsone, \tstwo]^\transpose$ are the limiting MLE parameters under the "wrong" structure $S^3$. Therefore, by the continuous mapping theorem, 
    $$\frac{f(\obsdata\mid S^3)}{f(\obsdata\mid S^{1*})} \xrightarrow{a.s.} 0.$$
    Finally, by the application of Slutsky's lemma and the continuous mapping theorem the posterior probability of $S^1$ converges to
    \begin{align*}
        \pi(S^{1*}\mid \obsdata) &= \frac{1}{1 + \frac{f(\obsdata\mid S^2)}{f(\obsdata\mid S^{1*})} + \frac{f(\obsdata\mid S^3)}{f(\obsdata\mid S^{1*})}} \xrightarrow{a.s.} \frac{1}{1 + \frac{\gamma^{-1} \# \pi(\theta^*\mid S^2)}{\pi(\theta^*\mid S^{1*})}}.
    \end{align*}
\end{proof}
\section{Proof of Theorem \ref{thm:non_ident_S3}}
\label{appendix:proof_thm2_nonindent}
\begin{proof}
    In the case of $\mathcal{M}^3 = \{\theta^*, S^{3*}\}$ with $\theta^* := [0, \tsone, \tstwo]^\transpose$ being the "true" generating model, the corresponding MLEs converge in probability as follows

    $$\hat{\theta}^1, \hat{\theta}^2, \hat{\theta}^3  \xrightarrow{a.s.}[0, \tsone, \tstwo]^\transpose,$$
    because $\mathbb{E}[X(1)X(2)] = \mathbb{E}[X(1)]\mathbb{E}[X(2)] = 0$, $\mathbb{E}[X(1)^2] = \tsone$ and $\mathbb{E}[X(2)^2] = \tstwo$.
    Using the Laplace approximation (Lemma \ref{lemma:laplace_regular}), we expand the log-marginal likelihoods. Note that model $S^3$ has dimension $d_3=2$, while $S^1$ and $S^2$ have dimension $d_1=d_2=3$. The Bayes factor between the complex models ($S^1, S^2$) and the true simple model ($S^3$) is determined by the difference in dimensions (Occam's razor penalty):

    \begin{equation}
        \log \frac{f(\obsdata\mid S^1)}{f(\obsdata\mid S^{3*})} = \ell_n(\hat{\theta}^1\mid  S^1) - \ell_n(\hat{\theta}^3\mid S^3) - \frac{3-2}{2}\log n + O(1) \text{ a.s}.
    \end{equation}

    Exponentiating this, we analyze the behavior of the likelihood ratio scaled by $\sqrt{n}$
    \begin{align*}
        \sqrt{n} \frac{f(\obsdata\mid S^1)}{f(\obsdata\mid S^{3*})} &= \exp\left(\ell_n(\hat{\theta}^1\mid S^1) - \ell_n(\hat{\theta}^3\mid S^3) \right) \cdot n^{1/2} \cdot n^{-1/2} \cdot e^{O_p(1)} = \lambda_1,    
    \end{align*}
    with $\lambda_1 = O_p(1)$ as we show next. Using the definitions introduced in Lemma \ref{lemma:log_lik_simple}, twice the log-likelihood difference is 

    $$2(\ell_n(\hat{\theta}^1\mid S^1) - \ell_n(\hat{\theta}^3\mid S^{3*})) =\frac{2n}{2}\log\Big(1-\frac{(S_{12}^{X})^2}{S_1^XS_2^X}\Big) = n\log(1-r_n^2).$$
    
    where for simplicity we denote $r_n = \frac{S_{12}^X}{\sqrt{S_1^XS_2^X}}$. By the strong law of large numbers, we have
    \begin{equation}
    \label{eq:derivation_chi1}
            \frac{1}{n}S_1^X \xrightarrow{a.s.}\tsone, \quad \frac{1}{n}S_2^X \xrightarrow{a.s.}\tstwo
    \end{equation}

    and by the central limit theorem 
    \begin{equation}
    \label{eq:derivation_chi2}
        \frac{1}{\sqrt{n}}S_{12}^X \xrightarrow{d}N(0, \tsone\tstwo).
    \end{equation}

    Therefore, using the continuous mapping theorem and then Slutsky's lemma, we obtain that $\sqrt{n}r_n \xrightarrow{p} N(0,1)$. To find the distribution of the statistics $n \log(1-r_n^2)$ we perform a Taylor expansion of the function $f(t) = \log(1-t)$ around $t=0$, substituting $t = r_n^2$
    
    $$n\log(1-r_n^2) = -nr_n^2-\frac{n(r_n^2)^2}{2} -O(nr_n^6) .$$
    We established previously that $n r_n^2 \xrightarrow{d} \chi^2_1$. Therefore, $-n r_n^2 \xrightarrow{d} -\chi^2_1$. We can rewrite $\frac{n r_n^4}{2}$ as $\frac{1}{2n} (n r_n^2)^2$. Since $n r_n^2$ converges to a random variable (is $O_p(1)$) and $\frac{1}{2n} \to 0$, the product converges to $0$ in probability, i.e. $
        \frac{n r_n^4}{2} \xrightarrow{p} 0.$ Therefore, by Slutsky's lemma, the limit is determined entirely by the leading term
        
    \begin{equation}
        \label{eq:derivation_chi3}
        n \log(1 - r_n^2) \xrightarrow{d} -\chi^2_1.
    \end{equation}   
Therefore,  $2(\ell_n(\hat{\theta}^1\mid S^1) - \ell_n(\hat{\theta}^3\mid S^{3*}))$ converges in distribution to a $\chi_1^2$ random variable and hence it is bounded in probability.
Similarly, for $S^2,$ we have $\sqrt{n} \frac{f(\obsdata\mid S^1)}{f(\obsdata\mid S^{3*})} = \lambda_2$ with $\lambda_2 = O_p(1)$. The posterior probability of $S^3$ is 
        $$\pi(S^{3*}\mid\obsdata) = \frac{1}{1+\frac{f(\obsdata\mid S^1)}{f(\obsdata\mid S^{3*})} + \frac{f(\obsdata\mid S^2)}{f(\obsdata\mid S^{3*})}}.$$

    For simplicity, let $C_n := \frac{f(\obsdata\mid S^1)}{f(\obsdata\mid S^3)} + \frac{f(\obsdata\mid S^2)}{f(\obsdata\mid S^3)}$ and $Z_n := \frac{1}{1+C_n}$. Since $\sqrt{n} C_n = O_p(1)$, it implies $C_n = O_p(n^{-1/2})$, and thus $C_n \xrightarrow{p} 0$. By the continuous mapping theorem, $Z_n \to 1$. To analyze the rate of convergence to 1, we investigate the behaviour of the scaled difference $\sqrt{n}(Z_n-1)$. Expanding the difference, we obtain

    $$\sqrt{n}(Z_n-1) = \sqrt{n}\Big(\frac{1}{1+C_n}-1\Big) = -\sqrt{n}C_n\times \frac{1}{1+C_n}.$$

    Since $\sqrt{n}C_n \xrightarrow{d} \lambda_1+\lambda_2$ and $C_n \xrightarrow{p} 0$ by Slutsky's lemma we obtain that

    $$\sqrt{n}(Z_n-1) \xrightarrow{d} -(\lambda_1+\lambda_2).$$
    Therefore, the posterior mass concentrates to 1 with an error of order $O_p(n^{-1/2})$.
\end{proof}

\section{Proof of Theorem \ref{thm:ident_S1}: \texorpdfstring{$S^{1*}$}{S1*} the correct model}
\label{appendix:proof_thm1_ident}

\begin{theorem}
    \label{thm:ident_S1_app}
    Let $\obsdata = (X_i)_{1\leq i \leq n_N}$ be $n_N$ i.i.d. samples generated by \eqref{eq:SCM_vector} and $\intdata = (Y_i)_{1\leq i\leq m_N}$ be $m_N$ i.i.d. samples generated by interventional model, when intervened upon the second node, $\cdot\mid do(Y(2) = y )$. Assume that the priors $\pi(\theta\mid S)$ are four-times differentiable for $S \in \mathcal{S}$. If the true generating model is $\mathcal{M}^1 = \{\theta^*, S^{1*}\}$ with $\theta^*:= [w^*, \tsone, \tstwo]^\transpose$, as $N \to \infty$, we have,

    \[
        \frac{1}{N} \log \left( \frac{1}{\pi(S^{1*} \mid \mixdata)} - 1 \right) \xrightarrow{a.s.} -D_{12}(\eta),
    \]
    with 
\begin{equation}
    \label{eq:D_12_def_app}
    D_{12}(\eta) := \frac{1}{2} \log \left( \frac{\eta \sigma_{1,X}^2 + (1-\eta) \sigma_{1,Y}^2}{(\sigma_{1,X}^2)^\eta (\tsone)^{\bar{\eta}}} \right)
\end{equation}
    and  $\sigma_{1,X}^2:=w^{*2}\tstwo+\tsone,$ $\sigma_{1, Y}^2 := w^{*2}y^2+\tsone$.

\begin{proof}
    Given that the covariance matrix for $X$ exists and $y \in \mathbb{R}$, the second moments defined in Lemma \ref{lemma:MLEs_ident} are finite. Let $\bar{\eta} := 1 - \eta$. Assuming that $S^{1*}$ is the true generating model, we apply the strong law of large numbers to the sufficient statistics of the mixed dataset.

    The MLE derived in Lemma \ref{lemma:MLEs_ident} under the first model, $S^{1*}$, is consistent:
    \[
        \hat{w}^1 = \frac{S_{12}^X+S_{12}^Y}{S_2^X+S_2^Y} = \frac{\eta_N \frac{S_{12}^X}{n} + \bar{\eta}_N \frac{S_{12}^Y}{m}}{\eta_N \frac{S_{2}^X}{n} + \bar{\eta}_N \frac{S_{2}^Y}{m}} 
        \xrightarrow{a.s.} \frac{\eta\mathbb{E}[X(1)X(2)] + \bar{\eta}y\mathbb{E}[Y(1)]}{\eta\mathbb{E}[X(2)^2] + \bar{\eta}y^2} = w^*.
    \]
    Similarly, \{$\hat{\tau}_1^{2}  \mid S^{1*}\}\xrightarrow{a.s.} \tau_1^{*2}$ and $\{\hat{\tau}_2^{2} \mid S^{1*}\}  \xrightarrow{a.s.} \tau_2^{*2}$.

    For the second structure $S^2$, assuming $S^{1*}$ is true, the MLE $\hat{\theta}^2$ converge to a pseudo-true vector $\theta^2_\infty$ as $N \to \infty$
    \begin{align*}
        \frac{S_{12}^X}{S_1^X} &\xrightarrow{a.s} \frac{w^*\tau_2^{*2}}{w^{*2}\tau_2^{*2} + \tau_1^{*2}}, \\
        \frac{S_1^X + S_1^Y}{N} &\xrightarrow{a.s.} \eta(w^{*2}\tau_2^{*2} + \tau_1^{*2}) + \bar{\eta}(w^{*2}y^2 + \tau_1^{*2}), \\
        \frac{S_1^X S_2^X - (S_{12}^X)^2}{n_N S_1^X} &\xrightarrow{a.s.} \frac{\tau_1^{*2}\tau_2^{*2}}{w^{*2}\tau_2^{*2} + \tau_1^{*2}}.
    \end{align*}
    Therefore, $\hat{\theta}^1 \xrightarrow{a.s.} \theta^*$ and $\hat{\theta}^2 \xrightarrow{a.s.} \theta^2_\infty$. We can explicitly decompose the limit $\theta^2_\infty$ into the observational non-identifiable component $\gamma(\theta^*)$ and a shift term induced by the intervention:
    \begin{equation}
        \label{eq:theta2_S1_nonident}
        \theta^2_\infty = \gamma(\theta^*) +  [0, \bar{\eta} w^{*2}(y^2 - \tau_2^{*2}), 0]^\transpose.
    \end{equation}
    The non-zero middle term, $\bar{\eta} w^{*2}(y^2 - \tau_2^{*2})$, represents the discrepancy between the natural variance of the cause and the fixed interventional value. This term breaks the observational symmetry provided $y^2 \neq \tau_2^{*2}$. Finally, for the third structure $S^3$ (independence, where $w=0$), the estimator for $w$ is constrained to zero, and the variance parameters converge to the marginal variances of the data mixture. Specifically, we have
    \begin{equation}
        \hat{\theta}^3 \xrightarrow{a.s.} \theta^3_\infty = [ 0, \eta(w^{*2}\tau_2^{*2} + \tau_1^{*2}) + \bar{\eta}(w^{*2}y^2 + \tau_1^{*2}),\tau_2^{*2}]^\transpose.
    \end{equation}

    The log-marginal likelihood of the mixed dataset $\mixdata$ under $S^{1*}$ is 
    \begin{align*}
    \log f(\mixdata \mid S^{1*}) &= \log \int_{\Theta} \exp(\ell_{n,m}(\theta\mid S^{1*})) \pi(\theta\mid S^{1*})d\theta\\
    &= \log\int_{\Theta} \exp\left(\ell_n(\theta \mid S^{1*}) + \ell_m(\theta \mid S^{1*})\right) \pi(\theta \mid S^{1*}) \, d\theta, \\
    \end{align*}
where we define $\ell_n(\theta\mid S)$ and $\ell_m(\theta\mid S)$ as being the corresponding log-likelihoods under structure $S$ of the observational, $\obsdata$, and interventional datasets, $\intdata$, respectively

\begin{equation}
\label{eq:def_ln_lm}
    \ell_n(\theta\mid S):= \log f(\obsdata\mid \theta, S)\quad \text{and} \quad \ell_m(\theta\mid S) := \log f(\intdata\mid \theta, S).
\end{equation}
    
Since the models satisfy the Laplace regularity conditions described in Lemma \ref{lemma:laplace_regular}, the marginal likelihood can be approximated as follows as $N\to \infty$

\begin{align}
\label{eq:laplace_approx_S1}
    \nonumber\log f(\mixdata \mid S^{1*}) &= \left( \ell_n(\hat{\theta}^1 \mid S^{1*}) + \ell_m(\hat{\theta}^1 \mid S^{1*}) \right) + \frac{3}{2}\log 2\pi \\
    & -\frac{1}{2} \log \det\left( \nabla^2 \ell_n(\hat{\theta}^1 \mid S^{1*}) + \nabla^2 \ell_m(\hat{\theta}^1 \mid S^{1*}) \right) + \log\left( \pi(\hat{\theta}^1 \mid S^{1*}) + O\left(\frac{1}{N}\right) \right) \text{ a.s.} 
\end{align}
    
Similarly for the remaining models $S^k$ ($k=2,3$) we have
    \begin{align*}
    \log f(\mixdata \mid S^{k}) &= \left( \ell_n(\hat{\theta}^k \mid S^{k}) + \ell_m(\hat{\theta}^k \mid S^{k}) \right) + \frac{d_k}{2}\log 2\pi \\
    & -\frac{1}{2} \log \det\left( \nabla^2 \ell_n(\hat{\theta}^k \mid S^{k}) + \nabla^2 \ell_m(\hat{\theta}^k \mid S^{k}) \right) + \log\left( \pi(\hat{\theta}^k \mid S^{k}) + O\left(\frac{1}{N}\right) \right) \text{ a.s.}
\end{align*}
where $d_k$ denotes the corresponding dimensions of the parameter spaces, i.e. $d_2 = 3$ and $d_3 =2$ (there are only 2 parameters for the independence model because $w$ is set to 0). 
Normalizing by $1/N$ the log-marginal likelihood ratio of $S^{1*}$ and $S^2$ we obtain, as $N\to \infty$

\begin{align}
    \label{eq:expansion_ident}
    \nonumber\frac{1}{N}\log \frac{f(\mixdata\mid S^{1*})}{f(\mixdata\mid S^2)} &= \eta_N \left[ \frac{1}{n_N} \left( \ell_n(\hat{\theta}^1 \mid S^{1*}) - \ell_n(\hat{\theta}^2 \mid S^2) \right) \right] + (1-\eta_N) \left[ \frac{1}{m_N} \left( \ell_m(\hat{\theta}^1 \mid S^{1*}) - \ell_m(\hat{\theta}^2 \mid S^2) \right) \right] \\
    & + O\left(\frac{\log N}{N}\right) \text{ a.s.}
\end{align}

Since both $\ell_n(\theta\mid S)$ and $\ell_m(\theta \mid S)$ are continuous for all $\theta\in \Theta$, then, by the continuous mapping theorem and the strong law of large numbers, the log-marginal likelihood odds converge to a weighted sum of relative entropies

    \begin{equation}
        \label{eq:mixture_KL1}
        \frac{1}{N} \log \frac{f(\mixdata \mid S^{1*})}{f(\mixdata \mid S^2)} \xrightarrow{a.s.} \eta D(P_{m(\theta^*, S^{1*})}^X \| P_{m(\theta^2_\infty, S^2)}^X) + \bar{\eta} D(P_{m(\theta^*, S^{1*})}^Y \| P_{m(\theta^2_\infty, S^2)}^Y) =: D_{12},
    \end{equation}
    and similarly
    \begin{equation}
    \label{eq:mixture_KL2}
        \frac{1}{N} \log \frac{f(\mixdata \mid S^{1*})}{f(\mixdata \mid S^3)} \xrightarrow{a.s.} \eta D(P_{m(\theta^*, S^{1*})}^X \| P_{m(\theta^3_\infty, S^3)}^X) + \bar{\eta} D(P_{m(\theta^*, S^{1*})}^Y \| P_{m(\theta^3_\infty, S^3)}^Y) =: D_{13}.
    \end{equation}
    with $D_{12}$ is defined in \eqref{eq:D_12_def}. Next, the difference between the two relative entropies is
    \begin{align*}
        D_{12} - D_{13} &= \frac{1}{2}\log \tau_1^{*2} - \frac{1-\eta}{2}\log \tau_1^{*2} - \frac{1-\eta}{2} \log(w^{*2}\tau_2^{*2} + \tau_1^{*2}) \\
        &= -\frac{\eta}{2} \log \frac{w^{*2}\tau_2^{*2} + \tau_1^{*2}}{\tau_1^{*2}} < 0,
    \end{align*}

    Since the posterior probability of $S^{1*}$ is given by the following formula,

    $$\pi(S^{1*}\mid \mixdata) = \frac{1}{1+\frac{f(\mixdata\mid S^2)}{f(\mixdata\mid S^{1*})}+ \frac{f(\mixdata\mid S^3)}{f(\mixdata\mid S^{1*})}}$$
    and $D_{12}$ and $D_{13}$ are strictly positive (within each model, the parameters are identifiable) with $D_{12}<D_{13}$ the log-posterior ratio for the true model behaves as follows
    \[
        \frac{1}{N} \log \left( \frac{1}{\pi(S^{1*} \mid \mixdata)} - 1 \right) \xrightarrow{a.s.} - \min(D_{12}, D_{13}) = -D_{12}.
    \]
    Thus, $\pi(S^{1*} \mid \mixdata) \xrightarrow{a.s.} 1$.
\end{proof}

\section{Proof of Theorem \ref{thm:ident_S1}: \texorpdfstring{$S^{2*}$}{S2*} the correct model}
\label{appendix:proof_thm2_ident}
\begin{theorem}
Let $\obsdata = (X_i)_{1\leq i \leq n_N}$ be $n_N$ i.i.d. samples generated by \eqref{eq:SCM_vector} and $\intdata = (Y_i)_{1\leq i\leq m_N}$ be $m_N$ i.i.d. samples generated by interventional model, when intervened upon the second node, $\cdot\mid do(Y(2) = y )$. Assume that the priors $\pi(\theta\mid S)$ are four-times differentiable for any $S \in \mathcal{S}$. If the true generating model is $\mathcal{M}^2 = \{\theta^*, S^{2*}\}$ with $\theta^* := [w^*, \tsone, \tstwo]^\transpose$ as $N \to \infty$, we have

    \[
        \frac{1}{N} \log \left( \frac{1}{\pi(S^{2*} \mid \mixdata)} - 1 \right) \xrightarrow{a.s.} -D_{21}(\eta),
    \]
    with 
    \begin{equation}
        \label{eq:D_21_def_app}
        D_{21}(\eta) = \frac{1}{2}\log \left( 1- \frac{\eta^2 w^{*2}\tau_1^{*2}}{\eta\sigma_{2, Y}^2 +\bar{\eta}y^2}\right) \left(\frac{\sigma_{2, Y}^2}{\tau_2^{*2}} \right)^ \frac{\eta}{2},
    \end{equation}
with $\sigma_{2, Y}^2 = w^{*2}\tsone+\tstwo$.
\end{theorem}
\end{theorem}

\begin{proof}
    We proceed analogously to the proof of Theorem \ref{thm:ident_S1_app}. Given that the covariance matrix of $X$ exists and the intervention value $y \in \mathbb{R}$ is fixed, all relevant second moments are finite. Let $\bar{\eta} := 1 - \eta$.
    
    First, we establish the consistency of the MLEs under the true generating model $S^{2*}$. Applying the strong law of large numbers to the sufficient statistics of the mixed dataset $\mixdata$ we have
    \begin{equation}
        \hat{w} = \frac{S_{12}^X + S_{12}^Y}{S_1^X + S_1^Y} \xrightarrow{a.s.} \frac{\eta \mathbb{E}[X(1)X(2)] + \bar{\eta}\mathbb{E}[Y(1)Y(2)]}{\eta \mathbb{E}[X(1)^2] + \bar{\eta}\mathbb{E}[Y(1)^2]}.
    \end{equation}
    Under the true causal structure $X(1) \to X(2)$, the intervention on the child node $X(2)$ does not alter the marginal distribution of the parent $X(1)$. Therefore, $\mathbb{E}[Y(1)Y(2)] = \mathbb{E}[X(1)]y = 0$ and $\mathbb{E}[Y(1)^2] = \tau_1^{*2}$. Consequently:
    \begin{equation}
        \hat{w} \xrightarrow{a.s.} \frac{\eta w^* \tau_1^{*2}}{\eta \tau_1^{*2} + \bar{\eta}\tau_1^{*2}} = w^*.
    \end{equation}
    Similarly, the variance estimators satisfy
    \begin{align*}
        \hat{\tau}_1^{2} &= \frac{S_1^X + S_1^Y}{N} \xrightarrow{a.s.} \eta \tau_1^{*2} + \bar{\eta} \tau_1^{*2} = \tau_1^{*2}, \\
        \hat{\tau}_2^{2} &= \frac{S_2^X + S_2^Y - \hat{w}^2(S_{12}^X + S_{12}^Y)}{N} \xrightarrow{a.s.} \tau_2^{*2}.
    \end{align*}
    Thus, $\hat{\theta}^2 \xrightarrow{a.s.} \theta^*$.

    Next, we analyze the convergence of the MLEs for the incorrect reverse model $S^1$. Given the data generating process $S^2$, the estimators converge to a pseudo-true parameter vector $\theta_\infty^1$, which is as follows
    \begin{equation}
        \theta_\infty^1 = \left[ 
        \frac{\eta w^*\tau_1^{*2}}{\eta(w^{*2}\tau_1^{*2}+\tau_2^{*2})+\bar{\eta}y^2}, \quad
        \frac{\eta\tau_1^{*2}\tau_2^{*2}+\eta\bar{\eta}w^{*2}\tau_1^{*4}+\bar{\eta}y^2\tau_1^{*2}}{\eta(w^{*2}\tau_1^{*2}+\tau_2^{*2})+\bar{\eta}y^2}, \quad
        w^{*2}\tau_1^{*2}+\tau_2^{*2} 
        \right]^\transpose.
        \label{eq:theta1_limit_under_M2}
    \end{equation}
    Notably, unlike the purely observational case, the presence of interventional data ($\eta < 1$) affects both the estimated weight and the variance of the first node, breaking the symmetry with $\gamma^{-1}(\theta^*)$. For $\eta \to 1$, the above expression recovers the observational limit.

    For the independence model $S^3$, the MLEs converge to:
    \begin{equation}
        \hat{\theta}^3 \xrightarrow{a.s.} [0, \tau_1^{*2}, w^{*2}\tau_1^{*2}+\tau_2^{*2}]^\transpose =: \theta^3_{\infty}.
    \end{equation}

    Following the same steps as in \eqref{eq:expansion_ident}-- \eqref{eq:mixture_KL2}, the normalized log-marginal likelihood ratio converge almost surely to the weighted relative entropies
    \begin{align*}
        \frac{1}{N} \log \frac{f(\mixdata \mid S^{2*})}{f(\mixdata \mid S^1)} &\xrightarrow{a.s.} \eta D(P_{m(\theta^*, S^{2*})}^X \| P_{m(\theta^1_\infty, S^1)}^X) + \bar{\eta} D(P_{m(\theta^*, S^{2*})}^Y \| P_{m(\theta^1_\infty, S^1)}^Y) =: D_{21}, \\
        \frac{1}{N} \log \frac{f(\mixdata \mid S^{2*})}{f(\mixdata \mid S^3)} &\xrightarrow{a.s.} \eta D(P_{m(\theta^*, S^{2*})}^X \| P_{m(\theta^3_\infty, S^3)}^X) + \bar{\eta} D(P_{m(\theta^*, S^{2*})}^Y \| P_{m(\theta^3_\infty, S^3)}^Y) =: D_{23}.
    \end{align*}
    Evaluating the Gaussian relative entropies yields
    \begin{align}
        D_{21} &= \frac{1}{2}\log \left( 1- \frac{\eta^2 w^{*2}\tau_1^{*2}}{\eta(w^{*2}\tau_1^{*2}+\tau_2^{*2}) +\bar{\eta}y^2}\right) + \frac{\eta}{2}\log\left(1+\frac{w^{*2}\tau_1^{*2}}{\tau_2^{*2}} \right), \\
        D_{23} &= \frac{\eta}{2}\log\left(1+\frac{w^{*2}\tau_1^{*2}}{\tau_2^{*2}} \right).
    \end{align}
    Comparing the two rates, we observe that
    \begin{equation*}
        D_{21} = D_{23} + \frac{1}{2}\log \left( 1- \frac{\eta^2 w^{*2}\tau_1^{*2}}{\eta(w^{*2}\tau_1^{*2}+\tau_2^{*2}) +\bar{\eta}y^2}\right).
    \end{equation*}
    Since the term inside the logarithm is strictly less than 1 because $\eta^2w^{*2}\tsone < \eta w^{*2}\tsone$ for any $\eta \in (0, 1)$, the log term is negative, implying $D_{21} < D_{23}$. This indicates that the reverse model $S^1$ is "closer" to the true distribution than the independence model $S^3$. Therefore, for any $\eta>0$, the posterior $\pi(S^{2*}\mid \mixdata)$ concentrates to 1 at an exponential rate given by the minimum relative entropy
    \begin{equation*}
        \frac{1}{N}\log\left(\frac{1}{\pi(S^{2*} \mid \mixdata)} - 1 \right) \xrightarrow{a.s.} - \min(D_{21}, D_{23}) = -D_{21}.
    \end{equation*}
\end{proof}

\section{Proof of Theorem \ref{thm:ident_S3}}

\begin{proof}
    If $S^{3*}$ is the true generating model, then it is easy to see from Lemma \ref{lemma:MLEs_ident} that the MLEs of all three models converge almost surely to the same limit by the strong law of large numbers, i.e.

    $$\hat{\theta}^1, \hat{\theta}^2, \hat{\theta}^3 \xrightarrow{a.s.} [0, \tsone, \tstwo]^\transpose = \theta^*$$
    because $\frac{1}{n}S_{12}^X \xrightarrow{a.s.} \mathbb{E}[X(1)X(2)] = \mathbb{E}[X(1)]\mathbb{E}[X(2)] = 0 $ and $\frac{1}{n}S_{12}^Y \xrightarrow{a.s.} y\mathbb{E}[Y(1)] = 0$. Since the models are all Laplace regular from Lemma \ref{lemma:laplace_regular}, the marginal likelihoods can be approximated for $N\to \infty$ as in \eqref{eq:laplace_approx_S1}
    
    \begin{align*}
    \log f(\mixdata \mid S^{1}) &= \left( \ell_n(\hat{\theta}^1 \mid S^{1}) + \ell_m(\hat{\theta}^1 \mid S^{1}) \right) + \frac{3}{2}\log 2\pi \\
    & -\frac{1}{2} \log \det\left( \nabla^2 \ell_n(\hat{\theta}^1 \mid S^{1*}) + \nabla^2 \ell_m(\hat{\theta}^1 \mid S^{1}) \right) + \log\left( \pi(\hat{\theta}^1 \mid S^{1}) + O\left(\frac{1}{N}\right) \right) a.s.
\end{align*}
    Next, by the continuous mapping theorem and the strong law of large numbers, we have that the two Hessians weighted by $\frac{1}{n_N}$ and $\frac{1}{m_N}$, respectively, converge almost surely to their respective negative Fisher information matrices as $N \to \infty$, i.e.

    \begin{align*}
        -\frac{1}{n_N}\nabla^2 \ell_n(\theta\mid  S^1) &\xrightarrow{a.s.} Diag\left(\frac{\tau_2^2}{\tau_1^2}, \frac{1}{2\tau_1^4}, \frac{1}{2\tau_2^4}\right) =: I^X_1([w, \tau_1^2, \tau_2^2]^\transpose)\\
        -\frac{1}{m_N}\nabla^2 \ell_m(\theta\mid S^1) &\xrightarrow{a.s.} Diag\left(\frac{y^2}{\tau_1^2}, \frac{1}{2\tau_1^4}, 0\right) =: I^Y_1([w, \tau_1^2, \tau_2^2]^\transpose)
    \end{align*}
    
    For completeness, we provide the Fisher information matrices for the remaining two models

    $$I_2^X([w, \tau_1^2, \tau_2^2]^\transpose) := Diag\left( \frac{\tau_1^2}{\tau_2^2}, \frac{1}{2\tau_2^4}, \frac{1}{2\tau_1^4}\right)\quad \text{and}\quad I_2^Y([w, \tau_1^2, \tau_2^2]^\transpose):= Diag\left(0,\frac{1}{2\tau_2^4}, 0\right)$$
    for the second structure $S^2$ and

    $$I_3^X([0, \tau_1^2, \tau_2^2]^\transpose) := Diag\left( \frac{1}{2\tau_1^4}, \frac{1}{2\tau_2^4}\right)\quad \text{and}\quad I_3^Y([0, \tau_1^2, \tau_2^2]^\transpose):=\left(\frac{1}{2\tau_1^4}, 0 \right)$$
    for the third structure, $S^3$. Therefore, by the continuous mapping theorem

    \begin{align}
        \label{eq:det_calc1}
        \nonumber \frac{1}{N^3} \det\left( \nabla^2 \ell_n(\hat{\theta}^1 \mid  S^1) + \nabla^2 \ell_m(\hat{\theta}^1 \mid  S^1) \right)  &=  \det\left( \eta_N\frac{1}{n_N}\nabla^2 \ell_n(\hat{\theta}^1  \mid S^1) + \bar{\eta}_N\frac{1}{m_N}\nabla^2 \ell_m(\hat{\theta}^1 \mid S^1) \right) \\
        & \xrightarrow{a.s.} \det \left( \eta I_1^X(\theta^*)+\bar{\eta}I_1^Y(\theta^*) \right) = \frac{\eta(\eta\tstwo+\bar{\eta}y^2)}{4\tau_1^{*6}\tau_2^{*4}},
    \end{align}
    
    \begin{align}
        \label{eq:det_calc2}
        \nonumber\frac{1}{N^3} \det\left( \nabla^2 \ell_n(\hat{\theta}^2 \mid  S^2) + \nabla^2 \ell_m(\hat{\theta}^2 \mid  S^2) \right) &= \det\left( \eta_N\frac{1}{n_N}\nabla^2 \ell_n(\hat{\theta}^2  \mid S^2) + \bar{\eta}_N\frac{1}{m_N}\nabla^2 \ell_m(\hat{\theta}^2 \mid S^2) \right)\\
        & \xrightarrow{a.s.} \det \left( \eta I_2^X(\theta^*)+\bar{\eta}I_2^Y(\theta^*) \right)  = \frac{\eta^2}{4\tau_2^{*6}\tsone}
    \end{align}
    and for the third model, $S^{3*}$
    \begin{align}
        \label{eq:det_calc3}
        \nonumber\frac{1}{N^2} \det\left( \nabla^2 \ell_n(\hat{\theta}^3 \mid  S^{3*}) + \nabla^2 \ell_m(\hat{\theta}^3 \mid  S^{3*}) \right) &= \det\left( \eta_N\frac{1}{n_N}\nabla^2 \ell_n(\hat{\theta}^3  \mid S^{3*}) + \bar{\eta}_N\frac{1}{m_N}\nabla^2 \ell_m(\hat{\theta}^3 \mid S^{3*}) \right)\\
        & \xrightarrow{a.s.} \det \left( \eta I_3^X(\theta^*)+\bar{\eta}I_3^Y(\theta^*) \right)   = \frac{\bar{\eta}}{4\tau_2^{*4}\tau_1^{*4}}.
    \end{align}
    
    Similar to the proof of Theorem \ref{thm:non_ident_S3}, the maximum log-likelihood difference between models $S^1$ and $S^3$ is 

    \begin{equation}
        \label{eq:log_lik_diff1-3}
        \ell_{n,m}(\hat{\theta}^3\mid S^{3*}) - \ell_{n,m}(\hat{\theta}^1\mid S^1) = \frac{n_N+m_N}{2}\log \left(1-\frac{(S_{12}^X+S_{12}^Y)^2}{(S_1^X+S_1^Y)(S_2^X+S_2^Y)} \right) + \frac{n_N}{2}\log \left(1-\frac{(S_{12}^X)^2}{S_1^XS_2^X} \right)
    \end{equation}
    with both terms in the sum converging in distribution to a $\chi^2_1$, i.e.

    $$-(n_N+m_N)\log \left(1-\frac{(S_{12}^X+S_{12}^Y)^2}{(S_1^X+S_1^Y)(S_2^X+S_2^Y)} \right) \xrightarrow{d} \chi_1^2$$
    and
    $$-n_N\log \left(1-\frac{(S_{12}^X)^2}{S_1^XS_2^X} \right) \xrightarrow{d} \chi_1^2$$
    as per the same steps, \eqref{eq:derivation_chi1}-- \eqref{eq:derivation_chi3}, in the proof of Theorem \ref{thm:non_ident_S3}. Next, the log-likelihood difference converges in distribution to a weighted $\chi^2_1$ distribution, and hence is bounded in probability, i.e. it is $O_p(1)$. Similarly, the log-likelihood difference between models $S^2$ and $S^{3*}$ is expressed as

    \begin{equation}
        \label{eq: log_lik_diff2-3}
        2\left(\ell_{n,m}(\hat{\theta}^3\mid  S^3) - \ell_{n,m}(\hat{\theta}^2\mid  S^2) \right) = \frac{n_N}{2}\log \left(1-\frac{(S_{12}^X)^2}{S_1^XS_2^X} \right) \xrightarrow{d} - \frac{1}{2}\chi^2_1, 
    \end{equation}
    which again is bounded in probability, $O_p(1)$. Then, using 
    \eqref{eq:log_lik_diff1-3}--\eqref{eq: log_lik_diff2-3}, the weighted log-marginal likelihood odds between models $S^1$ vs $S^{3*}$ and $S^2$ vs $S^{3*}$ are

    $$\sqrt{N}\log \frac{f(\mixdata\mid S^3)}{f(\mixdata\mid S^1)} \xrightarrow{d} \Lambda_1, \text{ and } \sqrt{N}\log \frac{f(\mixdata\mid S^3)}{f(\mixdata\mid S^2)} \xrightarrow{d} \Lambda_2$$ 
    with $\Lambda_1$ and $\Lambda_2$ being both $O_p(1)$, bounded in probability. Therefore, the posterior probability of the third model, $S^{3*}$, concentrates to 1 asymptotically at a rate of $O_p(1/\sqrt{N})$

    $$ \sqrt{N}(\pi(S^{3*}\mid \mixdata)-1) \xrightarrow{d} -(\Lambda_1+\Lambda_2).$$
\end{proof}

\section{Exact calculation of the posterior}

\begin{theorem}
    \label{thm:BGe_nonindent}
    Let $\obsdata = (X_i)_{1\leq i\leq n}$ be i.i.d. observations generated by \eqref{eq:SCM_vector}. Given the hierarchical priors defined in \eqref{eq:BGe_priors_compact}, the posterior distribution over the structures, represented by the vector $[\pi(S^1 \mid \obsdata), \pi(S^2 \mid \obsdata), \pi(S^3 \mid \obsdata)]^\transpose$, is
    \begin{align*}
        \Bigg[ 
        & \frac{(2\beta)^{\alpha_1+\alpha_2} \Gamma(\alpha_2+\frac{n}{2})\Gamma(\alpha_1+\frac{n}{2})}{\sqrt{\eta} \pi^n \Gamma(\alpha_1)\Gamma(\alpha_2)} \cdot \frac{(S_2^{X\lambda})^{\alpha_1+(n-1)/2}}{(S_2^{X\beta})^{\alpha_2+n/2}} \cdot \frac{1}{(S_1^{X\beta} S_2^{X\lambda} - (S_{12}^X)^2)^{\alpha_1+n/2}}, \\
        & \frac{(2\beta)^{\alpha_3+\alpha_4} \Gamma(\alpha_3+\frac{n}{2})\Gamma(\alpha_4+\frac{n}{2})}{\sqrt{\eta} \pi^n \Gamma(\alpha_3)\Gamma(\alpha_4)} \cdot \frac{(S_1^{X\lambda})^{\alpha_4+(n-1)/2}}{(S_1^{X\beta})^{\alpha_3+n/2}} \cdot \frac{1}{(S_1^{X\lambda} S_2^{X\beta} - (S_{12}^X)^2)^{\alpha_4+n/2}}, \\
        & \frac{(2\beta)^{\alpha_5+\alpha_6} \Gamma(\alpha_5+\frac{n}{2})\Gamma(\alpha_6+\frac{n}{2})}{\pi^n \Gamma(\alpha_5)\Gamma(\alpha_6)} \cdot \frac{1}{(S_1^{X\beta})^{\alpha_5+n/2}(S_2^{X\beta})^{\alpha_6+n/2}} 
        \Bigg]^\transpose,
    \end{align*}
    where 
    \begin{equation}
        S_i^{X\lambda} := S_{i}^X + \frac{1}{\lambda}, \quad S_{i}^{X\beta} := S_{i}^X + 2\beta \quad \text{for } i \in \{1, 2\},
        \label{eq:def_S_beta_eta}
    \end{equation}
    and $S_1^X, S_2^X$, and $S_{12}^X$ are defined in \eqref{eq:MLE_interv_X}.
\end{theorem}
\begin{proof}
    \begin{align}
        \notag \pi(S^1\mid\obsdata) &\propto \int f(\obsdata\mid S^1, w, \tau_1^2, \tau_2^2) f(w\mid\tau_1^2)f(\tau_1^2)f(\tau_2^2)\, d(\tau_1^2)\, d(\tau_2^2)\, dw \\
        \notag &= \int \frac{1}{(\sqrt{2\pi \tau_1^2})^n}\exp \left(-\frac{1}{2\tau_1^2}\sum_{i=1}^n(X_1^{i}-wX_2^{i})^2 \right) \frac{1}{(\sqrt{2\pi \tau_2^2})^n}\exp \left(-\frac{1}{2\tau_2^2}S_2^X \right) \\
        \notag &\quad \times \frac{1}{\sqrt{2\pi\lambda\tau_1^2}}\exp\left(-\frac{w^2}{2\lambda\tau_1^2} \right) \frac{\beta^{\alpha_1+\alpha_2}}{\Gamma(\alpha_1)\Gamma(\alpha_2)} (\tau_1^2)^{-\alpha_1-1} \exp \left(-\frac{\beta}{\tau_1^2} \right) (\tau_2^2)^{-\alpha_2-1} \exp \left(-\frac{\beta}{\tau_2^2} \right)\, d(\tau_1^2)\, d(\tau_2^2)\, dw \\
        \notag &= \frac{1}{\sqrt{\lambda}} \frac{1}{(2\pi)^{n+1/2}} \frac{\beta^{\alpha_1+\alpha_2}}{\Gamma(\alpha_1)\Gamma(\alpha_2)} \frac{\Gamma(\alpha_2+\frac{n}{2})}{\left( \beta + \frac{S_2^X}{2} \right)^{\alpha_2+n/2}}\times \int_{-\infty}^\infty \frac{\Gamma(\alpha_1+\frac{n+1}{2})}{\left( \beta+ \frac{\sum_{i=1}^n(X_1^{i}-wX_2^{i})^2}{2} +\frac{w^2}{2\lambda}   \right)^{\alpha_1+\frac{n+1}{2}}} \, dw \\
        \notag &= \frac{(2\beta)^{\alpha_1+\alpha_2} \Gamma(\alpha_2+\frac{n}{2})\Gamma(\alpha_1+\frac{n}{2})}{\sqrt{\lambda} \pi^n \Gamma(\alpha_1)\Gamma(\alpha_2)} \cdot \frac{(1/\lambda + S_2^X)^{\alpha_1+(n-1)/2}}{(2\beta + S_2^X)^{\alpha_2+n/2}} \cdot \frac{1}{[(2\beta+S_1^X)(1/\lambda+S_2^X) - (S_{12}^X)^2]^{\alpha_1+n/2}}\\
        &= \frac{(2\beta)^{\alpha_1+\alpha_2} \Gamma(\alpha_2+\frac{n}{2})\Gamma(\alpha_1+\frac{n}{2})}{\sqrt{\lambda} \pi^n \Gamma(\alpha_1)\Gamma(\alpha_2)} \cdot \frac{ (S_2^{X\lambda})^{\alpha_1+(n-1)/2}}{(S_2^{X\beta})^{\alpha_2+n/2}} \cdot \frac{1}{(S_1^{X\beta} S_2^{X\lambda} - (S_{12}^X)^2)^{\alpha_1+n/2}},
        \label{eq:post_G1}
    \end{align}

    Similarly, we have
    \begin{equation}
        \pi(S^2\mid\obsdata) = \frac{(2\beta)^{\alpha_3+\alpha_4} \Gamma(\alpha_3+\frac{n}{2})\Gamma(\alpha_4+\frac{n}{2})}{\sqrt{\lambda} \pi^n \Gamma(\alpha_3)\Gamma(\alpha_4)} \cdot \frac{ (S_1^{X\lambda})^{\alpha_4+(n-1)/2}}{(S_1^{X\beta})^{\alpha_3+n/2}} \cdot \frac{1}{(S_1^{X\lambda} S_2^{X\beta} - (S_{12}^X)^2)^{\alpha_4+n/2}},
        \label{eq:post_G2}
    \end{equation}
    and
    \begin{align}
        \notag \pi(S^3\mid\obsdata) &\propto \int f(\obsdata\mid S^3, \tau_1^2, \tau_2^2) f(\tau_1^2)f(\tau_2^2)\, d(\tau_1^2)\, d(\tau_2^2) \\
        \notag &= \int \frac{1}{(\sqrt{2\pi \tau_1^2})^n}\exp \left(-\frac{1}{2\tau_1^2}S_1^X \right) \frac{1}{(\sqrt{2\pi \tau_2^2})^n}\exp \left(-\frac{1}{2\tau_2^2}S_2^X \right)\\
        \notag &\quad \times \frac{\beta^{\alpha_5+\alpha_6}}{\Gamma(\alpha_5)\Gamma(\alpha_6)} (\tau_1^2)^{-\alpha_5-1} \exp \left(-\frac{\beta}{\tau_1^2} \right) (\tau_2^2)^{-\alpha_6-1} \exp \left(-\frac{\beta}{\tau_2^2} \right)\, d(\tau_1^2)\, d(\tau_2^2)\\
        &= \frac{(2\beta)^{\alpha_5+\alpha_6} \Gamma(\alpha_5+\frac{n}{2})\Gamma(\alpha_6+\frac{n}{2})}{ \pi^n \Gamma(\alpha_5)\Gamma(\alpha_6)} \cdot \frac{1}{ (S_1^{X\beta})^{\alpha_5+n/2}(S_2^{X\beta})^{\alpha_6+n/2}}.
        \label{eq:post_G3}
    \end{align}
\end{proof}

\begin{theorem}
    \label{thm:BGe_interv}
    Let $\obsdata := (X_i)_{1\leq i \leq n}$ be $n$ i.i.d. observational samples, $\intdata:=(Y_i)_{1\leq i \leq m}$ be $m$ i.i.d. interventional samples from $\cdot \mid \text{do}(Y(2)= y)$ for a fixed $y\in \mathbb{R}$ and $\mixdata := \obsdata \cup \intdata$. Given the hierarchical priors defined in \eqref{eq:BGe_priors_compact}, the posterior distribution over the structures, $\left[\pi(S^1\mid \mixdata), \pi(S^2\mid \mixdata), \pi(S^3\mid \mixdata)) \right]^\transpose$ is given by

    \begin{align}
    \notag\frac{1}{c} \Bigg[&\frac{(2\beta)^{\alpha_1+\alpha_2}\Gamma(\alpha_1+\frac{m+n}{2})\Gamma(\alpha_2+\frac{n}{2})}{\sqrt{\eta}\pi^{n+\frac{m+1}{2}}\Gamma(\alpha_1)\Gamma(\alpha_2)}\frac{(S_2^{X\lambda}+S_2^Y)^{\alpha_1+\frac{m+n-1}{2}}}{[(S_1^{X\beta}+S_1^Y)(S_2^{X\lambda}+S_2^Y) - (S_{12}^X+S_{12}^Y)^2]^{\alpha_1+\frac{m+n}{2}} (S_2^{X\beta})^{\alpha_2+\frac{n}{2}}},\\
    \notag &\frac{(2\beta)^{\alpha_3+\alpha_4} \Gamma(\alpha_3+\frac{m+n}{2})\Gamma(\alpha_4+\frac{n}{2})}{\sqrt{\eta} \pi^{n+\frac{m+1}{2}} \Gamma(\alpha_3)\Gamma(\alpha_4)} \cdot \frac{(S_1^{X\lambda})^{\alpha_4+\frac{n-1}{2}}}{(S_1^{X\lambda}S_2^{X\beta} - (S_{12}^X)^2)^{\alpha_4+\frac{n}{2}}(S_1^{X\beta}+S_1^Y)^{\alpha_3+\frac{m+n}{2}}}, \\
    & \frac{(2\beta)^{\alpha_5+\alpha_6} \Gamma(\alpha_5+\frac{m+n}{2})\Gamma(\alpha_6+\frac{n}{2})}{\pi^{n+\frac{m+1}{2}}\Gamma(\alpha_5)\Gamma(\alpha_6)} \frac{1}{(S_1^{X\beta}+S_1^Y)^{\alpha_5+\frac{m+n}{2}}( S_2^{X\beta})^{\alpha_6+\frac{n}{2}}}
    \Bigg],
    \end{align}
    where $c$ is the normalization constant, $S_i^{X\beta}, S_i^{X\lambda}$ are defined in \eqref{eq:def_S_beta_eta} and $S_i^Y, S_{12}^Y$ are defined in \eqref{eq:MLE_interv_Y} with $i \in \{1, 2\}$.
\end{theorem}

\section{Further plots}

\begin{figure}[h!]
    \centering
    \includegraphics[width=\linewidth]{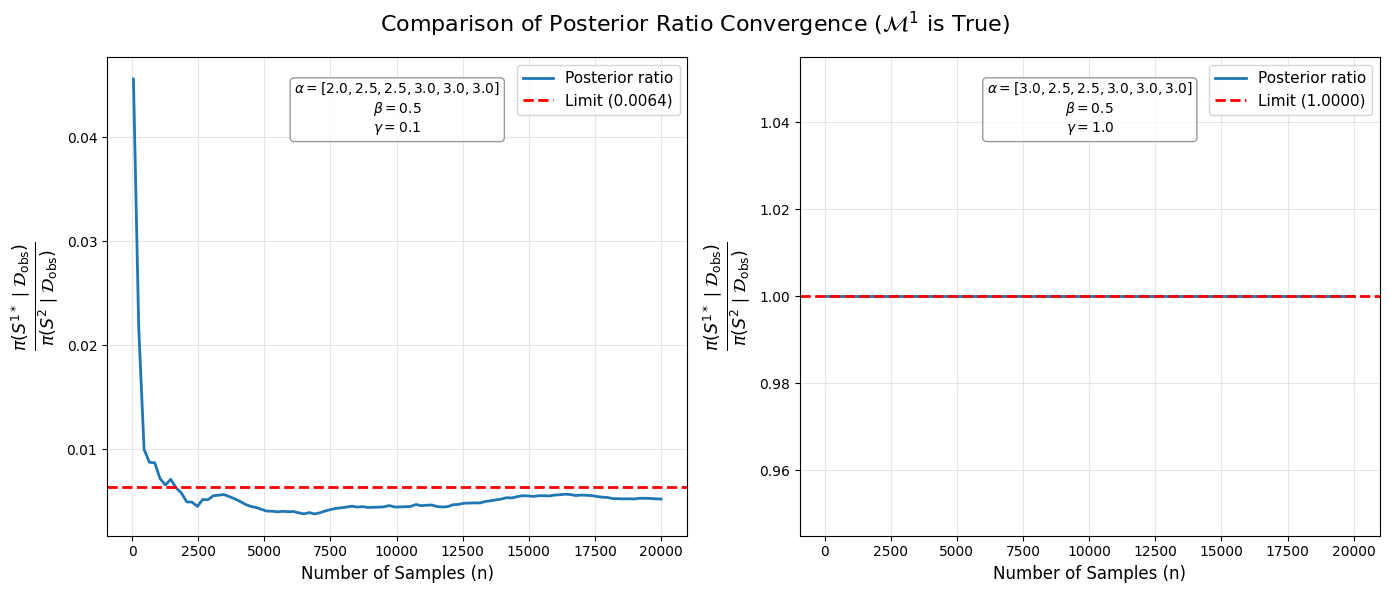}
    \caption{Posterior ratio $\frac{\pi(S^1 \mid \obsdata)}{\pi(S^2 \mid \obsdata)}$ as a function of sample size $n$ when the true model is $S^1$. The hyper-parameters in the right plot correspond to the BGe prior.}
    \label{fig:comparison_post_odds}
\end{figure}

\begin{figure}[!t]
    \centering
    \includegraphics[width=\linewidth]{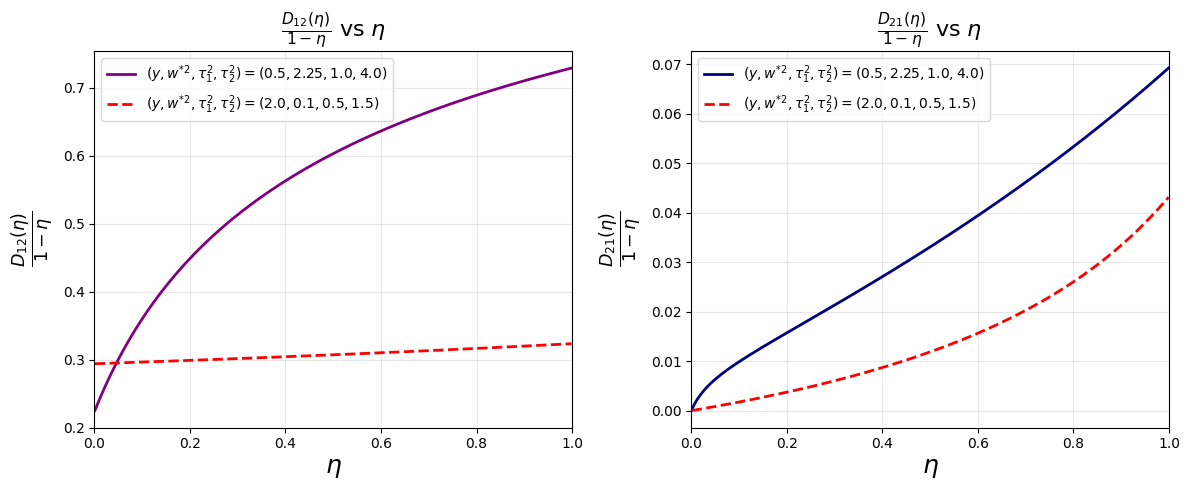}
    \caption{Comparison of the rates $\frac{D_{12}(\eta)}{1-\eta}$ (left) and $\frac{D_{21}(\eta)}{1-\eta}$ as  function of $\eta$ for different choices of $[w^*, \tsone, \tstwo]^\transpose$}
    \label{fig:new_exponents}
\end{figure}

\end{document}